\documentclass[11pt]{amsart}

\usepackage[a4paper,margin=1.05in]{geometry}
\usepackage{amsmath,amssymb,amsthm,mathtools}
\usepackage{enumitem}
\usepackage{float}
\usepackage{microtype}
\usepackage{placeins}
\usepackage[pagewise]{lineno}
\usepackage{xcolor}
\usepackage{tikz}
\usetikzlibrary{arrows.meta,calc,positioning}
\usepackage[colorlinks=true,
  linkcolor=blue!55!black,
  citecolor=blue!55!black,
  urlcolor=blue!55!black]{hyperref}
\usepackage[nameinlink,capitalise,noabbrev]{cleveref}

\hypersetup{
  pdftitle={Critical-Exponent Spectra and Rank-Two Inverse Realization on Biregular Trees},
  pdfauthor={Sanghoon Kwon},
  pdfsubject={Entropy spectra, fixed-rank stratification, and inverse realization on biregular trees},
  pdfkeywords={volume entropy, topological entropy, critical exponent,
    biregular tree, non-backtracking dynamics, weighted pressure,
    fixed-rank spectrum, inverse realization}
}

\setlist{itemsep=2pt,topsep=5pt}
\makeatletter
\def\paragraph{\@startsection{paragraph}{4}%
  \z@{.65\linespacing\@plus.45\linespacing}%
  {.25\linespacing}%
  {\normalfont\bfseries}}
\makeatother

\theoremstyle{plain}
\newtheorem{theorem}{Theorem}[section]
\newtheorem{proposition}[theorem]{Proposition}
\newtheorem{lemma}[theorem]{Lemma}
\newtheorem{corollary}[theorem]{Corollary}

\theoremstyle{definition}
\newtheorem{definition}[theorem]{Definition}
\newtheorem{problem}[theorem]{Problem}

\theoremstyle{remark}
\newtheorem{remark}[theorem]{Remark}

\newcommand{\T}{\mathcal T}
\newcommand{\Aut}{\operatorname{Aut}}
\newcommand{\Core}{\operatorname{Core}}
\newcommand{\cA}{\mathcal A}
\newcommand{\cP}{\mathcal P}
\newcommand{\ZZ}{\mathbb Z}
\newcommand{\QQ}{\mathbb Q}
\title[Critical-exponent spectra on biregular trees]
{Critical-Exponent Spectra and Rank-Two Inverse Realization
on Biregular Trees}

\author{Sanghoon Kwon}
\address{Department of Mathematics Education, Catholic Kwandong University,
Gangneung 25601, Republic of Korea}
\email{skwon@cku.ac.kr, shkwon1988@gmail.com}
\date{22 August, 2026}

\subjclass[2020]{Primary 37B40; Secondary 20E08, 37B10, 05C50, 20F69}
\keywords{volume entropy, topological entropy, critical exponent,
biregular tree, non-backtracking dynamics, weighted pressure,
fixed-rank spectrum}
\thanks{This work was supported by the Basic Science Research Program through
the National Research Foundation of Korea (NRF), funded by the Ministry of
Education (No.~RS-2025-25415913).}

\begin{document}
\raggedbottom
%\linenumbers

\begin{abstract}
We study the critical-exponent, or equivalently entropy, spectrum arising
from free type-preserving actions on the biregular tree
\(\mathcal T_{r+1,s+1}\).  For an action with a nonempty finite quotient
core, the critical exponent agrees with both the volume entropy of the
universal cover of the core and the topological entropy of the associated
non-backtracking edge shift.  The unrestricted spectrum is
\([0,\frac12\log(rs)]\), whereas the finitely generated spectrum is a
countable dense subset of this interval obtained by taking logarithms of
Hashimoto spectral radii.  We stratify this finite-state spectrum by the
circuit rank of the quotient core.  At each fixed rank, finitely many typed
kernels parametrize all values.  Moreover, every positive entropy value has
only finitely many kernel--length realizations, up to type-preserving
isomorphism, and every positive accumulation point of a fixed-rank spectrum
belongs to a lower-rank stratum.  At rank two, the corresponding exponential
rates admit a complete inverse classification in terms of three explicit
polynomial families, together with a finite exact membership test for
algebraic-integer inputs.
\end{abstract}

\maketitle
\enlargethispage{4pt}

\tableofcontents

% =============================================================================
\section{Introduction}
% =============================================================================

Entropy spectra record the possible exponential complexities of a dynamical
or geometric class under structural constraints.  In the present setting,
the same exponential rate has three equivalent interpretations: orbit growth
for a group acting on a tree, volume growth in a universal cover, and
topological entropy for a finite non-backtracking edge shift.  The purpose of
this paper is not only to determine the possible rates, but to describe how
their topology changes when finite generation and circuit rank are fixed.

Let \(X\) be a locally finite tree, let \(\Aut(X)\) carry the compact-open
topology, and let \(\Gamma<\Aut(X)\) be discrete.  For the graph metric \(d\)
and a vertex \(o\in X\), the critical exponent is
\begin{equation}\label{eq:critical-exponent-intro}
  \delta_\Gamma
  =\limsup_{R\to\infty}\frac1R
    \log\#\{\gamma\in\Gamma:d(o,\gamma o)\leq R\}.
\end{equation}
It is independent of \(o\).  For an infinite group it is also the exponent
of convergence of the Poincar\'e series
\(\sum_{\gamma\in\Gamma}e^{-t d(o,\gamma o)}\); for a finite group we use
the orbit-growth convention \(\delta_\Gamma=0\).

For a connected locally finite graph \(Y\), its volume entropy is
\begin{equation}\label{eq:volume-entropy-intro}
  h(Y)
  =\limsup_{R\to\infty}\frac1R\log\#B_R(y),
  \qquad B_R(y)=\{v\in V(Y):d(y,v)\leq R\}.
\end{equation}
Changing \(y\) changes radii by only a bounded amount, so \(h(Y)\) is
well-defined.

Fix integers \(r,s\geq2\).  The tree \(\T_{r+1,s+1}\) is bipartite, with
vertices of type zero having degree \(r+1\) and vertices of type one having
degree \(s+1\).  An automorphism is type-preserving if it preserves these
parts.  Direct sphere counting gives
\[
  h(\T_{r+1,s+1})=\frac12\log(rs),
\]
and therefore
\begin{equation}\label{eq:intro-biregular-bound}
  0\leq\delta_\Gamma\leq\frac12\log(rs).
\end{equation}
The regular tree is the specialization
\(\T_{q+1}=\T_{q+1,q+1}\).

A free action, meaning one with trivial vertex stabilizers and no edge
inversions, is an ordinary graph covering.  Its quotient is bipartite when
the action is type-preserving.  Removing hanging trees from the quotient
leaves its cyclic core.  For a finite core \(C\), let \(\vec E(C)\) be its
set of oriented edges and let \(B_C\) be the Hashimoto matrix on those
edges: an entry is one precisely when two
oriented edges concatenate without immediate reversal.  We write
\(\rho(M)\) for the spectral radius of a square matrix.  If the quotient has
a nonempty finite core \(C\) and \(\widetilde C\) denotes its universal cover,
the covering dictionary gives
\begin{equation}\label{eq:intro-growth-dictionary}
  e^{\delta_\Gamma}=\rho(B_C),
  \qquad
  \delta_\Gamma=\log\rho(B_C)
  =h(\widetilde C)=h_{\mathrm{top}}(\Sigma_C,\sigma_C),
\end{equation}
where \((\Sigma_C,\sigma_C)\) is the two-sided subshift of finite type whose
states are the oriented edges of \(C\) and whose allowed transitions are the
non-backtracking turns.  More generally, extend a positive integral
edge-length function to oriented edges by \(\ell(\bar f)=\ell(f)\), and let
\(P\) denote topological pressure on \((\Sigma_C,\sigma_C)\).  The entropy
\(h\) of the corresponding suspension flow is selected by the pressure
equation
\begin{equation}\label{eq:intro-pressure-dictionary}
  P(-h\ell)=0
  \quad\Longleftrightarrow\quad
  \rho\!\left(B_C\operatorname{diag}
    (e^{-h\ell(f)}:f\in\vec E(C))\right)=1.
\end{equation}
This weighted finite-state viewpoint is the mechanism behind the fixed-rank
finiteness and rank-drop arguments and the inverse formulas below.

\paragraph{Notation and spectral strata}
Throughout, \(\mathbb N=\{1,2,3,\ldots\}\).  We keep logarithmic and
multiplicative rates distinct:
\[
  \lambda_\Gamma=e^{\delta_\Gamma},
  \qquad
  \eta(C)=\log\rho(B_C).
\]
Define the unrestricted multiplicative spectrum
\[
  \mathcal S_{r,s}^{\mathrm{free}}
  =\{\lambda_\Gamma:\Gamma<\Aut(\T_{r+1,s+1})
       \text{ acts freely, discretely, and type-preservingly}\},
\]
and let \(\cA_{r,s}\) be the subset obtained by requiring \(\Gamma\) to be
finitely generated.  We call \(\cA_{r,s}\) the \emph{finite-state
spectrum}: each of its values is encoded by the finite Hashimoto matrix of a
finite quotient core, as proved in \cref{thm:intro-stratification}.  For
\(g\geq2\), let \(\mathcal R_{g;r,s}\) consist of
the elements \(\rho(B_C)>1\) of \(\cA_{r,s}\) realized by a core with
\[
  b_1(C)=|E(C)|-|V(C)|+1=g,
\]
where \(E(C)\) denotes the unoriented edge set.  Thus
\[
  \mathcal S_{r,s}^{\mathrm{free}}
  \supseteq\cA_{r,s}
  \supseteq\mathcal R_{g;r,s},
\]
and the paper determines how the spectrum changes along these successive
complexity restrictions.  All derived-set and fixed-rank statements are made
in the multiplicative variable \(\lambda=e^\delta\); the corresponding
additive entropy statements follow by applying \(\log\).

\paragraph{Main results}
The main structural result is the fixed-rank stratification and its rank-drop
boundary, stated in \cref{thm:intro-fixed-rank}.  We first identify the
unrestricted and finite-state levels of \cref{fig:global-roadmap} in order to
fix the ambient spectrum.

\begin{theorem}[Global and finite-state spectra]
\label{thm:intro-stratification}
Let \(r,s\geq2\).
\begin{enumerate}[label=\textup{(\roman*)}]
  \item One has
  \[
    \mathcal S_{r,s}^{\mathrm{free}}=[1,\sqrt{rs}].
  \]
  More precisely, for every
  \(\delta\in(0,\frac12\log(rs)]\), there is a connected bipartite
  \((r+1,s+1)\)-biregular quotient graph whose universal-cover deck group
  acts freely and type-preservingly with critical exponent \(\delta\).
  Its attached finite blocks are connected non-cyclic bipartite cores whose
  logarithmic rates increase strictly to \(\delta\) along both directions of
  a bi-infinite spine.  The trivial group supplies \(\delta=0\).

  \item The finitely generated stratum is
  \begin{equation}\label{eq:Ars-intro}
    \cA_{r,s}
    =\left\{\rho(B_C):
      \begin{array}{l}
        V(C)=C_0\sqcup C_1\text{ is a finite connected bipartite core},\\
        2\leq\deg(v)\leq r+1\ (v\in C_0),\\
        2\leq\deg(w)\leq s+1\ (w\in C_1)
      \end{array}\right\}.
  \end{equation}

  \item The set \(\cA_{r,s}\) is countable and dense in
  \([1,\sqrt{rs}]\).  A cycle contributes the value \(1\).
\end{enumerate}
\end{theorem}

\begin{figure}[H]
\centering
\begin{tikzpicture}[
  y=.9cm,
  box/.style={draw=blue!65!black,fill=blue!5,rounded corners=3pt,
    minimum height=.82cm,text width=5.2cm,align=center,font=\small,
    line width=.8pt},
  infinite/.style={box,draw=green!48!black,fill=green!5},
  finite/.style={box,draw=orange!72!black,fill=orange!7},
  arr/.style={-{Stealth[length=5pt]},line width=.9pt,blue!62!black},
  greenarr/.style={arr,green!48!black},
  orangearr/.style={arr,orange!72!black},
  tag/.style={font=\scriptsize,text=gray!72!black,fill=white,inner sep=1.2pt}
]
  \node[box] (free) at (0,0)
    {free type-preserving quotients of \(\T_{r+1,s+1}\)};
  \node[infinite] (all) at (-3.5,-1.65)
    {unrestricted quotient complexity\\full continuum};
  \node[finite] (fg) at (3.5,-1.65)
    {finite generation\\finite cores and Hashimoto radii};
  \node[finite] (rank) at (3.5,-3.25)
    {fixed rank \(b_1\!=\!g\)\\finitely many typed kernels};
  \node[finite] (ranktwo) at (3.5,-5.25)
    {rank two\\three kernels and inverse\\polynomial families};
  \node[infinite] (boundary) at (-3.5,-3.25)
    {degeneration boundary\\lower-rank spectral strata};

  \draw[greenarr] (free.south) -- ++(0,-.35) -| (all.north);
  \draw[orangearr] (free.south) -- ++(0,-.35) -| (fg.north);
  \draw[orangearr] (fg)--node[tag,right]{\hspace{0.5em} fix \(b_1=g\)} (rank);
  \draw[orangearr] (rank)--node[tag,right]{\hspace{0.5em} set \(g=2\)} (ranktwo);
  \draw[greenarr] (rank.west)--node[tag,above=1mm]{rank drop} (boundary.east);
\end{tikzpicture}
\caption{The spectral hierarchy.  Infinite quotient complexity yields the
full continuum; finite generation produces finite-state Hashimoto data;
fixed rank leaves finitely many typed kernels, whose degenerations land in
lower-rank strata; rank two has an explicit inverse classification.}
\label{fig:global-roadmap}
\end{figure}
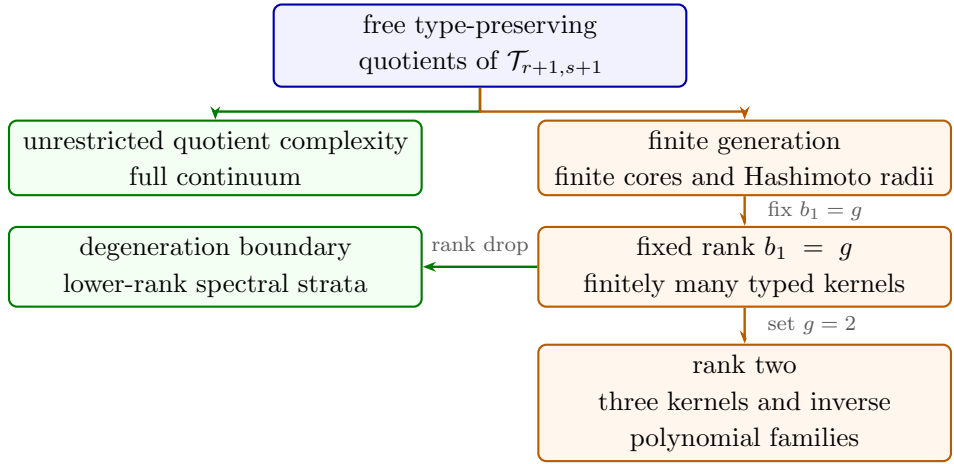

Finite generation therefore collapses the continuum to a countable dense
finite-state spectrum.  Suppressing maximal paths whose internal vertices
have degree two produces a typed kernel \((K,\tau)\), where
\(\tau:V(K)\to\mathbb Z/2\mathbb Z\) records the inherited bipartite vertex
type.  If a kernel edge \(e\) has endpoints \(v,w\), its suppressed path
length is admissible precisely when
\(\ell(e)\equiv\tau(v)+\tau(w)\pmod2\).  Thus the length data form an
admissible vector \(\ell\in\mathbb N^{E(K)}\).  Denote the Hashimoto radius of the
reconstructed subdivision \(K(\ell)\) by \(\alpha_{K,\ell}\).
For a subset \(A\subseteq\mathbb R\), we write \(A'\) for its derived
set, that is, the set of accumulation points of \(A\).

\begin{theorem}[Fixed-rank entropy stratification and rank drop]
\label{thm:intro-fixed-rank}
Fix \(g\geq2\) and \(r,s\geq2\).
\begin{enumerate}[label=\textup{(\roman*)}]
  \item There is a finite set \(\mathcal K_{g;r,s}\) of typed kernels such
  that
  \[
    \mathcal R_{g;r,s}
    =\bigcup_{(K,\tau)\in\mathcal K_{g;r,s}}
      \{\alpha_{K,\ell}:\ell\text{ is admissible for }\tau\}.
  \]
  Every such kernel satisfies
  \(|V(K)|\leq2g-2\) and \(|E(K)|\leq3g-3\).

  \item For every prescribed \(\alpha>1\), only finitely many typed
  kernel--length data of rank \(g\), up to type-preserving isomorphism,
  realize \(\alpha\).

  \item If pairwise distinct \(\alpha_n\in\mathcal R_{g;r,s}\) converge
  to \(\alpha>1\), then
  \(\alpha\in\mathcal R_{g_0;r,s}\) for some \(2\leq g_0<g\).  Equivalently,
  \begin{equation}\label{eq:intro-derived-set}
    \bigl(\mathcal R_{g;r,s}\bigr)'\cap(1,\infty)
    \subseteq\bigcup_{2\leq h<g}\mathcal R_{h;r,s}.
  \end{equation}
  In particular, the only accumulation point of
  \(\mathcal R_{2;r,s}\) in \([1,\infty)\) is \(1\).
\end{enumerate}
\end{theorem}

At rank two the excess-degree identity leaves only the figure-eight, theta,
and dumbbell kernels.

\begin{theorem}[Biregular rank-two inverse realization]
\label{thm:intro-rank-two}
Fix \(r,s\geq2\), and let \(\alpha>1\).  There is a finite connected bipartite core \(C\) with
\(b_1(C)=2\), with the degree bounds in \eqref{eq:Ars-intro}, and with
\(\rho(B_C)=\alpha\) if and only if at least one of the following holds:
\begin{enumerate}[label=\textup{(\roman*)}]
  \item \(\max\{r,s\}\geq3\), and for some even \(a,b\geq2\), the number
  \(\alpha\) is the unique root \(x>1\) of
  \[
    x^{a+b}-x^a-x^b-3=0;
  \]
  \item for some positive integers \(a,b,c\) of the same parity,
  \(\alpha\) is the unique root \(x>1\) of
  \[
    x^{a+b+c}-x^a-x^b-x^c-2=0;
  \]
  \item for some even \(a,b\geq2\) and some \(c\geq1\), \(\alpha\) is the
  unique root \(x>1\) of
  \[
    x^{2c}(x^a-1)(x^b-1)-4=0.
  \]
\end{enumerate}
The three cases are the typed figure-eight, theta, and dumbbell kernels.
Moreover, if \(\alpha\) is a real algebraic integer specified by its
minimal polynomial and a rational isolating interval, membership is decidable
by a finite exact search.
\end{theorem}

A real algebraic integer \(\lambda>1\) is a weak Perron number if every
Galois conjugate \(\lambda'\) satisfies \(|\lambda'|\leq\lambda\).  Every
element of \(\cA_{r,s}\setminus\{1\}\) is weak Perron, but this condition
does not record the bipartition, the reversal involution, or the two local
degree bounds.  The typed-kernel parametrization and the rank-two theorem
give an exact partial converse.

\begin{corollary}[Regular-tree specialization]\label{cor:intro-regular}
Set \(r=s=q\).  Then
\[
  \{\delta_\Gamma:\Gamma<\Aut(\T_{q+1})\text{ is discrete}\}
  =[0,\log q],
  \qquad
  \cA_{q,q}=\cA_q,
\]
where \(\cA_q\) is the finitely generated free-action spectrum on
\(\T_{q+1}\), without a type-preserving hypothesis.  Hence the regular
full-interval and unstratified finite-core statements follow from the
biregular results.
\end{corollary}

Indeed, a non-bipartite finite regular quotient may be replaced by its
canonical bipartite double cover without changing its universal cover or
its Hashimoto spectral radius.

\paragraph{Relation to previous work}
For a finite metric graph, volume entropy is equivalently characterized as
the topological entropy of its geodesic flow, and weighted edge lengths lead
to a finite pressure equation; see Lim \cite{Lim2008}.  The variation of
graph volume entropy under elementary modifications was studied by
Kim--Lim \cite{KimLim2020}.  Minimal entropy at fixed cyclomatic number has
also been investigated; a recent elementary treatment is due to Sabourau
\cite{Sabourau2026}.  Those works optimize entropy over metric or
combinatorial graphs.  The present work instead determines the spectra of
attainable values under two typed degree bounds, analyzes their fixed-rank
derived sets, and solves the rank-two inverse-realization problem.

The covering dictionary, the Hashimoto determinant, and the Ihara--Bass
formula are classical \cite{SerreTrees,Hashimoto1989,Bass1992,
StarkTerras1996,Terras2010}.
In the regular case, Kwon proved that every
\(\delta\in[0,\frac12\log q]\) is realized by an inversion-free discrete
subgroup of \(\Aut(\T_{q+1})\), using edge-indexed quotient graphs of groups
\cite{Kwon2019}.  This gives an earlier interval-realization result closely
related to the present full-interval theorem.  Here the class of actions is
restricted to free type-preserving actions, the ambient tree is allowed to be
biregular, and the spectrum is
further stratified by finite-state complexity and circuit rank.

The density of Hashimoto spectral radii of finite graphs under a single upper
degree bound follows from the configuration-model work of
Louvaris--Wise--Yehuda and from Tim\'ar's deterministic subdivision proof
\cite{LouvarisWiseYehuda2026,Timar2026}.  Louvaris--Wise--Yehuda also
construct subgroups of a free group with every growth rate by gluing
finite blocks along an infinite path with sufficiently long connecting arcs
\cite{LouvarisWiseYehuda2025EveryRate}.  More generally,
Coulon--Louvaris--Wise--Yehuda prove a full subgroup growth spectrum for
convex-cocompact free groups acting on proper hyperbolic spaces
\cite{CLWY2026}.  Applying their theorem to the cocompact deck action of
the free lattice
\(\pi_1(K_{s+1,r+1})<\Aut(\T_{r+1,s+1})\), where
\(K_{s+1,r+1}\) is the complete bipartite graph, recovers the unstratified
interval in logarithmic scale.  Thus the unrestricted interval by itself also
follows from their more general theorem.  We nevertheless include a direct,
self-contained typed construction in order to complete the stratified picture.
That theorem is not used below.

In another nearby direction, Lederle constructs boomerang subgroups of
a free group whose critical exponents are arbitrarily close to that of a
prescribed finitely generated subgroup \cite{Lederle2025}.  This approximation
problem is distinct from the exact typed realization and fixed-rank
stratification studied here.

Puder's extension of the cogrowth formula to arbitrary subsets of
regular and biregular trees concerns the relation between simple and
non-backtracking random-walk decay, rather than the subgroup
critical-exponent spectrum considered here \cite{Puder2024}.  This direction
is separate from the arguments below.  The proofs are self-contained relative
to these results: the typed density lemma below gives an adaptation of Tim\'ar's
subdivision argument, and the countable transfer estimate supplies a
self-contained sparse-gluing proof.

The main contributions are the finite-kernel
parametrization at fixed rank, finiteness at a prescribed rate, and the
rank-drop description \eqref{eq:intro-derived-set}.  They identify the
degeneration boundary of a finite-state entropy stratum.  At rank two,
explicit Ihara zeta calculations
are known \cite{KwonLee2020,ChicoMattmanRichards2025}; we use weighted
pressure to select the Perron root, incorporate the two vertex types and
their degree constraints, and obtain a complete and effective inverse test.
The unrestricted part supplies, in addition, an explicit typed realization
with separate degree bounds, degree-two attachment vertices (``ports'') in
both bipartite types, and a countable transfer-kernel estimate.  This
separates the previously known existence statements and zeta computations
from the stratified and typed inverse results of the paper.

\paragraph{Proof architecture and organization}
\Cref{sec:dictionary} establishes the covering and non-backtracking
dictionary.  In \cref{sec:stratification}, a parity-preserving typed density
lemma supplies finite blocks.  Sparse gluing gives the unrestricted
continuum, whereas finite generation gives the finite-core Hashimoto
spectrum and its arithmetic restrictions.  \Cref{sec:inverse} suppresses
degree-two chains, proves fixed-rank finiteness and rank drop, and derives
the complete rank-two inverse theorem.  The final section isolates the
unrestricted inverse problem and the derived-set questions suggested by
the rank stratification.

% =============================================================================
\section{Tree actions and non-backtracking finite-state dynamics}
\label{sec:dictionary}
% =============================================================================

Graphs are allowed to have parallel edges.  Loops occur only after kernel
suppression or in the untyped discussion.  A loop contributes two to the
vertex degree and has two oriented states.

\subsection{Critical exponents and ambient entropy}

Let a group \(\Gamma\) act by isometries on a locally finite tree \(X\).
For \(o\in X\), write
\[
  N_\Gamma(R;o)
  =\#\{\gamma\in\Gamma:d(o,\gamma o)\leq R\}.
\]
If the action is proper, then this number is finite for every \(R\), and we
define
\[
  \delta_\Gamma(X)
  =\limsup_{R\to\infty}\frac1R\log N_\Gamma(R;o).
\]
The notation \(\delta_\Gamma\) will be used when the ambient tree is fixed.

\begin{lemma}\label{lem:basepoint}
The critical exponent is independent of the base vertex.  If \(\Gamma\) is
a discrete subgroup of \(\Aut(X)\), then vertex stabilizers are finite, and
counting elements or distinct orbit points gives the same critical exponent.
\end{lemma}

\begin{proof}
For vertices \(o,o'\in X\), the triangle inequality gives
\[
  \bigl|d(o,\gamma o)-d(o',\gamma o')\bigr|
  \leq2d(o,o').
\]
The two counting functions therefore differ only by a bounded shift of the
radius.  A vertex stabilizer in \(\Aut(X)\) is compact and open.  Its
intersection with a discrete subgroup is finite, so the orbit map has finite
fibers of uniformly bounded cardinality.
\end{proof}

If \(\Gamma\) is infinite, the Poincar\'e series
\[
  \cP_\Gamma(t)
  =\sum_{\gamma\in\Gamma}\exp\!\bigl(-t d(o,\gamma o)\bigr)
\]
converges for \(t>\delta_\Gamma\) and diverges for
\(0\leq t<\delta_\Gamma\).  Thus its abscissa of convergence on
\([0,\infty)\) agrees with the orbit-counting definition.  For finite
\(\Gamma\), the series is finite for every real \(t\), while our
orbit-counting convention assigns critical exponent zero.

Fix integers \(r,s\geq2\).  The biregular tree \(\T_{r+1,s+1}\) is the
unique bipartite tree with vertex partition
\[
  V(\T_{r+1,s+1})=V_0\sqcup V_1
\]
such that vertices in \(V_0\) have degree \(r+1\), while vertices in \(V_1\)
have degree \(s+1\).  For an integer \(j\geq0\), let
\(S_j(o)=\{v:d(o,v)=j\}\) be the sphere of radius \(j\), and for real
\(R\geq0\), let \(B_R(o)=\{v:d(o,v)\leq R\}\) be the closed ball.  If
\(o\in V_0\), direct counting gives, for \(n\geq1\),
\begin{align*}
  \#S_{2n}(o)&=(r+1)s(rs)^{n-1},\\
  \#S_{2n+1}(o)&=(r+1)(rs)^n.
\end{align*}
Hence the volume entropy is
\begin{equation}\label{eq:biregular-entropy}
  h(\T_{r+1,s+1})
  =\lim_{R\to\infty}\frac1R\log\#B_R(o)
  =\frac12\log(rs).
\end{equation}
The regular tree \(\T_{q+1}\) is \(\T_{q+1,q+1}\), so its entropy is
\(\log q\).

For every \(R\geq0\), the orbit points counted up to displacement
\(R\) lie in \(B_R(o)\). Since the vertex stabilizer
\(\Gamma_o:=\{\gamma\in\Gamma:\gamma o=o\}\) is finite,
\[
  N_\Gamma(R;o)\leq |\Gamma_o|\,\#B_R(o).
\]
We therefore obtain the universal bound
\begin{equation}\label{eq:ambient-bound}
  0\leq\delta_\Gamma(\T_{r+1,s+1})
  \leq\frac12\log(rs).
\end{equation}

If an abstract free group acts with trivial vertex stabilizers, then it has
no edge inversions: an inversion would have order two, whereas a free group
is torsion-free.

If \(r\neq s\), every automorphism of \(\T_{r+1,s+1}\) preserves the two
types.  In the regular case, type preservation is an additional condition.

\subsection{Covering graphs and reduced loops}

Let \(A\) be a connected graph, fix a base vertex \(o\in A\), and let
\(p:\widetilde A\to A\) be its universal covering.  The fundamental group
\(\pi_1(A,o)\) acts freely on \(\widetilde A\) by deck transformations.  In
the other direction, if a group \(\Gamma\) acts freely and without inversion
on a tree \(X\), then
\[
  X\longrightarrow \Gamma\backslash X
\]
is a graph covering and
\[
  \Gamma\cong\pi_1(\Gamma\backslash X).
\]
We use this covering-space dictionary throughout; see
\cite{SerreTrees,Stallings1983} for the corresponding viewpoints from
groups acting on trees and finite graph immersions.
For the broader graph-of-groups and tree-lattice framework, which also
allows nontrivial stabilizers, see Bass--Lubotzky \cite{BassLubotzky2001}.

If \(A\) is bipartite with vertex partition \(V(A)=A_0\sqcup A_1\), and
every vertex in \(A_0\), respectively \(A_1\), has degree \(r+1\),
respectively \(s+1\), then
\(\widetilde A\cong\T_{r+1,s+1}\).  This elementary observation will allow us
to construct actions by constructing quotient graphs.

For a graph \(A\), let \(\vec E(A)\) be its set of oriented edges.  If
\(e\in\vec E(A)\), write \(o(e)\), \(t(e)\), and \(\bar e\) for its origin,
terminus, and reverse.  An edge path \(e_1\cdots e_n\) is
\emph{non-backtracking} if
\[
  e_{i+1}\neq\bar e_i\qquad(1\leq i<n).
\]
For graph paths this is equivalent to being reduced.

\begin{proposition}[Orbit--loop correspondence]\label{prop:orbit-loop}
Let \(A\) be a connected based graph, let \(\widetilde o\) be a lift of its
base vertex, and identify the deck group with \(\pi_1(A,o)\).  For every
\(n\geq0\), there is a bijection between
\begin{enumerate}[label=(\roman*)]
  \item elements \(\gamma\in\pi_1(A,o)\) such that
    \(d(\widetilde o,\gamma\widetilde o)=n\), and
  \item reduced closed paths of length \(n\) in \(A\) based at \(o\).
\end{enumerate}
Consequently, the critical exponent of the deck action is the logarithm of
the exponential growth rate of based non-backtracking loops in \(A\).
\end{proposition}

\begin{proof}
Every based homotopy class of paths in a graph has a unique reduced
representative.  The lift of this representative from \(\widetilde o\) is the
geodesic segment from \(\widetilde o\) to the corresponding deck translate.
Reduction and geodesicity therefore preserve length.
\end{proof}

\subsection{Cores}

\begin{definition}
The based core \(\Core(A,o)\) is the smallest connected based subgraph of
\(A\) containing every reduced closed path based at \(o\).  After deleting a
possibly trivial initial stem from the basepoint, its nontrivial vertices
have degree at least two.  A closed path \(e_1\cdots e_n\) is
\emph{cyclically reduced} if it is reduced and
\(e_1\neq\overline{e_n}\).  The \emph{cyclic core} \(\Core(A)\) is the union
of the images of all cyclically reduced closed paths in \(A\), with the base
vertex allowed to vary; it is empty when \(A\) is a tree.  This union contains
both the embedded cycles and the bridge paths needed to travel between them,
but excludes hanging stems, and it is connected whenever it is nonempty.
Thus \(\Core(A,o)\) is obtained from \(\Core(A)\) by adjoining a shortest path
from \(o\) to \(\Core(A)\), whenever the cyclic core is nonempty.
\end{definition}

Removing a hanging tree does not change the fundamental group.  It also does
not change the exponential growth of reduced based loops, since a reduced
closed path cannot enter a tree attached at a single vertex and return.

\begin{lemma}[Finite core criterion]\label{lem:finite-core}
Let \(A\) be a connected locally finite graph.  Then \(\pi_1(A)\) is finitely
generated if and only if \(\Core(A)\) is finite.
\end{lemma}

\begin{proof}
Choose a spanning tree \(S\subset A\).  The edges of \(A\setminus S\) form a
free basis of \(\pi_1(A)\), after adjoining the unique connecting paths in
\(S\).  Thus finite generation implies that only finitely many such edges
occur.  Their union with the finitely many connecting paths in \(S\) is a
finite connected subgraph carrying every cycle of \(A\); consequently the
cyclic core is finite.  Conversely, if \(\Core(A)=\varnothing\), then \(A\)
is a tree and \(\pi_1(A)\) is trivial.  If \(\Core(A)\neq\varnothing\),
deleting the hanging trees gives an isomorphism
\(\pi_1(A)\cong\pi_1(\Core(A))\), and the fundamental group of the finite
graph \(\Core(A)\) is finitely generated.
\end{proof}

\subsection{The Hashimoto matrix and topological entropy}

For a finite graph \(C\), its non-backtracking or Hashimoto matrix is the
\(\{0,1\}\)-matrix indexed by \(\vec E(C)\), defined by
\begin{equation}\label{eq:hashimoto}
  (B_C)_{e,f}
  =\begin{cases}
      1,&t(e)=o(f)\text{ and }f\neq\bar e,\\
      0,&\text{otherwise}.
    \end{cases}
\end{equation}
Powers of \(B_C\) count non-backtracking paths.  If \(C\) is connected, has
minimum degree at least two, and is not a cycle, then \(B_C\) is irreducible.
The two-sided non-backtracking edge shift is
\[
  \Sigma_C=\{(e_n)_{n\in\ZZ}: (B_C)_{e_n,e_{n+1}}=1\},
  \qquad \sigma_C((e_n)_{n\in\ZZ})=(e_{n+1})_{n\in\ZZ}.
\]
It is a subshift of finite type with adjacency matrix \(B_C\), and hence
\begin{equation}\label{eq:hashimoto-topological-entropy}
  h_{\mathrm{top}}(\Sigma_C,\sigma_C)=\log\rho(B_C);
\end{equation}
see, for example, \cite{Lind1984}.  Thus the spectral radius used throughout
is exactly the exponential of a topological entropy, not merely an auxiliary
linear-algebraic invariant.
Further finite-dimensional spectral properties of the non-backtracking
matrix, including information detected by its spectrum, are developed by
Glover--Kempton \cite{GloverKempton2021}.

\begin{proposition}\label{prop:finite-core-growth}
Let \(C\) be a finite connected core graph.  The exponential growth rate of
reduced paths in its universal cover issuing from a fixed vertex is
\(\rho(B_C)\); equivalently, this is the growth rate of non-backtracking paths
in \(C\) with a fixed admissible initial state.  If \(C\) is a cycle, this
value is \(1\).  Hence a deck group with finite quotient core \(C\) satisfies
\begin{equation}\label{eq:delta-rhoB}
  \delta_\Gamma=\log\rho(B_C).
\end{equation}
\end{proposition}

Related spectral properties of non-backtracking operators on universal
covers are developed in \cite{AngelFriedmanHoory2015}.
Eisner--Hoory further study the relation between the growth rate of a
universal covering tree and the entropy rate associated with non-backtracking
random walk \cite{EisnerHoory2026}.

\begin{proof}
For a non-cyclic core, irreducibility and Perron--Frobenius theory show that
the limsup exponential rate of every admissible entry of \(B_C^n\), and of
every nonzero finite sum of such entries, is \(\rho(B_C)\).  This formulation
allows for the period of \(B_C\).  More explicitly, if the base vertex
\(o\) lies in \(C\), then the number of non-backtracking paths of length
\(n\geq1\) starting at \(o\) is
\[
  \sum_{\substack{e\in\vec E(C)\\o(e)=o}}
  \sum_{f\in\vec E(C)}
  (B_C^{n-1})_{e,f},
\]
and therefore has limsup exponential rate \(\rho(B_C)\).  Similarly, the
number of reduced based loops of length \(n\geq1\) is
\[
  \sum_{\substack{e\in\vec E(C)\\o(e)=o}}
  \sum_{\substack{f\in\vec E(C)\\t(f)=o}}
  (B_C^{n-1})_{e,f}.
\]
Its limsup exponential rate is therefore \(\rho(B_C)\).  If the chosen base
vertex of the quotient
lies outside \(C\), \cref{lem:basepoint} allows us to move it to \(C\); the
discarded initial stem contributes only a bounded additive change in length.
The orbit--loop correspondence then gives \eqref{eq:delta-rhoB}.  A cycle has
two deterministic non-backtracking directions and spectral radius one.
\end{proof}

We shall also need a weighted form of the same counting argument before the
kernel reduction of \cref{sec:inverse}.  If
\(\ell:E(C)\to\mathbb N\) assigns a positive integral length to every
unoriented edge, let \(C(\ell)\) be the graph obtained by replacing each edge
\(e\) by a path of length \(\ell(e)\), and put
\begin{equation}\label{eq:early-weighted-turn}
  W_{C,\ell}(x)
  =B_C\operatorname{diag}\bigl(x^{-\ell(f)}:
      f\in\vec E(C)\bigr),
  \qquad x>1,
\end{equation}
where \(\ell(\bar f)=\ell(f)\).

\begin{lemma}[Weighted subdivision pressure]
\label{lem:weighted-subdivision-pressure}
Let \(C\) be a finite connected non-cyclic core.  There is a unique
\(\alpha_{C,\ell}>1\) such that
\begin{equation}\label{eq:early-weighted-pressure}
  \rho\bigl(W_{C,\ell}(\alpha_{C,\ell})\bigr)=1,
\end{equation}
and
\[
  \alpha_{C,\ell}=\rho(B_{C(\ell)}).
\]
If \(\ell'\geq\ell\) coordinatewise and \(\ell'\neq\ell\), then
\begin{equation}\label{eq:strict-subdivision-pressure}
  \rho(B_{C(\ell')})<\rho(B_{C(\ell)}).
\end{equation}
In particular, subdividing any nonempty set of edges of a finite connected
non-cyclic core strictly decreases the entropy of its universal cover.
\end{lemma}

Equivalently, \(\log\alpha_{C,\ell}\) is the topological entropy of the
suspension of the non-backtracking edge shift with roof function \(\ell\),
and \eqref{eq:early-weighted-pressure} is the pressure equation
\(P(-h\ell)=0\); compare \cite{Lim2008}.

\begin{proof}
Cut a non-backtracking path in \(C(\ell)\) whenever it reaches an original
vertex of \(C\).  Its successive original edges form a non-backtracking path
in \(C\), and traversing the next oriented edge \(f\) contributes length
\(\ell(f)\).  Conversely, every such weighted path in \(C\) has a unique
expansion in \(C(\ell)\).  Allowing the two endpoints to lie at subdivision
vertices adds only one of finitely many initial and terminal subpaths.  Hence
the weighted path series converges exactly when
\(\rho(W_{C,\ell}(x))<1\).  Consequently, the unique value of \(x\) for which
the weighted spectral radius equals one is the non-backtracking growth rate
of \(C(\ell)\).

The support of \(W_{C,\ell}(x)\) is the irreducible support of \(B_C\).
At \(x=1\), irreducibility and
\(B_C\mathbf 1\geq\mathbf 1\), where \(\mathbf1\) is the all-ones vector,
with strict inequality in at least one
coordinate because \(C\) is not a cycle, give \(\rho(B_C)>1\).  As
\(x\to\infty\), all entries tend to zero.  Strict Perron--Frobenius
comparison shows that
\(x\mapsto\rho(W_{C,\ell}(x))\) is continuous and strictly decreasing.
This proves existence, uniqueness, and the first assertion.

Now let \(\alpha=\alpha_{C,\ell}\).  If \(\ell'\geq\ell\) and at least one
length increases, then
\[
  0\leq W_{C,\ell'}(\alpha)
  \leq W_{C,\ell}(\alpha),
  \qquad
  W_{C,\ell'}(\alpha)\neq W_{C,\ell}(\alpha),
\]
while both matrices have the same irreducible support.  Thus
\(\rho(W_{C,\ell'}(\alpha))<1\).  Since the new weighted spectral radius is
strictly decreasing in \(x\), its value-one point is smaller than \(\alpha\),
which proves \eqref{eq:strict-subdivision-pressure}.
\end{proof}

If \(V(C)=C_0\sqcup C_1\) is a bipartition with the biregular degree bounds from
\eqref{eq:Ars-intro}, two non-backtracking steps have at most \(rs\)
continuations, and therefore
\begin{equation}\label{eq:biregular-rho-bound}
  1\leq\rho(B_C)\leq\sqrt{rs}.
\end{equation}
Setting \(r=s=q\) gives, as a regular-tree corollary,
\begin{equation}\label{eq:regular-rho-bound}
  1\leq\rho(B_C)\leq q
\end{equation}
whenever \(2\leq\deg_C(v)\leq q+1\).

% =============================================================================
\section{Global and finite-state entropy spectra}
\label{sec:stratification}

This section proves \cref{thm:intro-stratification}.  We first construct
finite typed cores whose rates approximate a prescribed target.  We then
place these cores along a parity-controlled bi-infinite spine and choose the
connecting corridors so long that mixed block transitions do not change the
target exponent.  The same finite cores give the exact finitely generated
spectrum.  The logical order of the unrestricted construction is summarized
in \cref{fig:sparse-proof-roadmap}.

\begin{figure}[H]
\centering
\begin{tikzpicture}[
  flow/.style={draw=blue!62!black,fill=blue!4,rounded corners=3pt,
    minimum height=.78cm,text width=4.05cm,align=center,font=\small,
    line width=.8pt},
  arr/.style={-{Stealth[length=5pt]},line width=.9pt,blue!68!black},
  lab/.style={font=\scriptsize,text=gray!70!black,fill=white,inner sep=1.2pt}
]
  \node[flow] (target) {target exponent\\ \(\delta\)};
  \node[flow,right=1.1cm of target] (cores)
    {double-ported cores\\ rates \mbox{increasing} to \(\delta\)};
  \node[flow,right=1.1cm of cores] (spine)
    {typed spine\\ prescribed corridor parities};
  \node[flow,below=1.1cm of spine] (corridors)
    {long typed corridors\\ summable transitions};
  \node[flow,left=1.1cm of corridors] (action)
    {free quotient action\\ \(\delta_\Gamma=\delta\)};

  \draw[arr] (target)--node[lab,above=0.5mm]{density} (cores);
  \draw[arr] (cores)--node[lab,above=0.5mm]{assign} (spine);
  \draw[arr] (spine)--node[lab,right=0.5mm]{separate} (corridors);
  \draw[arr] (corridors)--node[lab,above=0.5mm]{transfer} (action);
\end{tikzpicture}
\caption{Proof roadmap for the unrestricted free stratum.  The first arrow
supplies intrinsic block growth; the remaining arrows control how the blocks
are assembled without changing the target exponent.}
\label{fig:sparse-proof-roadmap}
\end{figure}
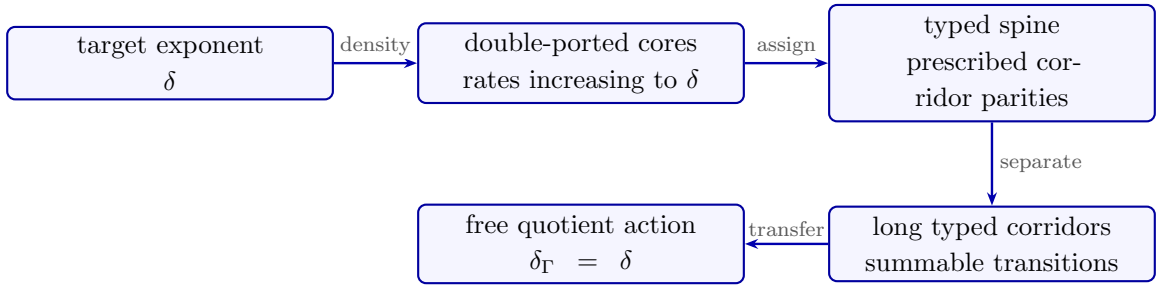

\subsection{Typed finite-core density}

A \emph{double-ported typed core} is a triple \((C,p_0,p_1)\), where
\(C=C_0\sqcup C_1\) is a finite connected non-cyclic bipartite core,
\[
  \deg_C(v)\leq r+1\quad(v\in C_0),
  \qquad
  \deg_C(v)\leq s+1\quad(v\in C_1),
\]
and \(p_\varepsilon\in C_\varepsilon\) has degree two for
\(\varepsilon=0,1\).  Put
\[
  \eta(C)=\log\rho(B_C)=h(\widetilde C),
\]
in accordance with the notation fixed in the introduction.

For edges \(e,f\) of a graph \(Y\), let \(d_E(e,f)\) denote their distance
in the line graph of \(Y\).  A set \(\mathcal E\subseteq E(Y)\) is
\emph{\(L\)-separated} if
\(d_E(e,f)\geq L+1\) whenever \(e\neq f\) belong to \(\mathcal E\).

\begin{lemma}[Separated subdivision estimate]
\label{lem:separated-subdivision}
Let \(T\) be a bounded-degree tree and put
\(\lambda(T)=\exp(h(T))\).  Let \(\mathcal E\subseteq E(T)\) be
\(L\)-separated, and let \(T_{\mathcal E}\) be obtained by subdividing every
edge of \(\mathcal E\) once.  Then
\begin{equation}\label{eq:separated-subdivision-estimate}
  \lambda(T)^{(L+1)/(L+2)}
  \leq \lambda(T_{\mathcal E})
  \leq \lambda(T).
\end{equation}
In particular, if \(\lambda(T)\leq\Lambda\), then
\begin{equation}\label{eq:separated-subdivision-log-change}
  0\leq h(T)-h(T_{\mathcal E})
  \leq\frac{\log\Lambda}{L+2}.
\end{equation}
\end{lemma}

\begin{proof}
Use an original vertex as basepoint.  A geodesic of length \(m\) in \(T\)
contains at most
\[
  1+\frac{m}{L+1}
\]
edges of \(\mathcal E\).  The corresponding path in
\(T_{\mathcal E}\) therefore has length at most
\[
  m+1+\frac{m}{L+1}
  =\frac{L+2}{L+1}m+1.
\]
Comparison of balls and passage to the limsup in the definition of volume
entropy gives the lower bound in
\eqref{eq:separated-subdivision-estimate}.  Conversely, contraction of the
inserted degree-two vertices does not increase distance.  Assigning every
new vertex to an adjacent original vertex and using the degree bound shows
that a ball in \(T_{\mathcal E}\) has at most a fixed multiple of the
corresponding ball in \(T\).  This proves the upper bound.  Taking logarithms
and using \(h(T)\leq\log\Lambda\) gives
\eqref{eq:separated-subdivision-log-change}.
\end{proof}

\begin{lemma}[Finite covers with separated lifted color classes]
\label{lem:separated-lift-coloring}
Let \(G\) be a finite connected graph and let \(D\geq1\).  There are a
finite connected cover \(H\to G\) and a finite coloring of \(E(H)\) such
that, for the universal covering map \(p:\widetilde H\to H\), the set
\[
  p^{-1}\{e\in E(H):e\text{ has color }c\}
\]
is \(D\)-separated in \(\widetilde H\) for every color \(c\).
If \(G\) is bipartite, then so is \(H\).
\end{lemma}

\begin{proof}
If \(G\) is a tree, take \(H=G\) and give every edge a different color.
Otherwise \(\pi_1(G)\) is a finitely generated free group.  Residual
finiteness supplies a finite-index normal subgroup which contains no element
conjugate to a nontrivial cyclically reduced closed path of length at most
\(4D+10\): there are only finitely many such conjugacy classes, and one
intersects finite-index normal subgroups avoiding one representative from
each.  This is the standard residual-finiteness construction for a
finitely generated free group; equivalently, it may be expressed using
finite graph covers as in \cite{Stallings1983}.  If the corresponding finite
connected cover \(H\to G\) contained a closed reduced circuit of length at
most \(4D+10\), its projection would represent one of the excluded
conjugacy classes.  Hence \(H\) has girth greater than \(4D+10\).  Two
distinct lifts inside a ball of radius \(D+2\) would project to the same
vertex only if their joining path projected to a nontrivial reduced circuit
of length at most \(2(D+2)<4D+10\).  Thus the universal covering map is
injective on every such ball.

The vertices of the line graph of \(H\) are the edges of \(H\).  Properly
color the \(D\)-th power of this finite line graph, where two vertices are
adjacent when their line-graph distance is at most \(D\).  Consider two
distinct lifted edges \(\widetilde e,\widetilde f\) of the same color.  If
\(p(\widetilde e)\neq p(\widetilde f)\), a line-graph path of length at most
\(D\) upstairs would project to one downstairs, contradicting the coloring.
If \(p(\widetilde e)=p(\widetilde f)\) and their line-graph distance were at
most \(D\), both lifted edges would lie in a ball of radius \(D+2\) on which
\(p\) is injective, again a contradiction.  Thus every lifted color class is
\(D\)-separated.  Finally, every cover of a bipartite graph is bipartite.
\end{proof}

\begin{lemma}[Typed double-ported density]
\label{lem:ported-density}
For \(r,s\geq2\), the numbers \(\eta(C)\) obtained from double-ported typed
cores are dense in \((0,\frac12\log(rs))\).  Consequently, for each
\(\delta\in(0,\frac12\log(rs)]\) there is a sequence
\((\widehat C_n,\widehat p_{n,0},\widehat p_{n,1})\) of such cores satisfying
\begin{equation}\label{eq:ported-approximation}
  0<\eta(\widehat C_n)<\delta,
  \qquad
  \eta(\widehat C_n)\nearrow\delta.
\end{equation}
\end{lemma}

\begin{proof}
Put \(h_0=\frac12\log(rs)\).  We make
the separated-subdivision argument of Tim\'ar \cite{Timar2026} quantitative
through \cref{lem:separated-subdivision} and retain the two vertex types.
Begin with the complete bipartite graph \(K_{s+1,r+1}\), whose universal
cover is \(\T_{r+1,s+1}\), and replace every edge by a path of a positive odd
length \(M\).  Call the resulting finite bipartite graph \(G_M\).  Its
universal-cover entropy is \(h_0/M\), and the original endpoint types extend
uniquely because \(M\) is odd.

Fix \(L\geq1\).  Apply \cref{lem:separated-lift-coloring} to \(G_M\) with
\(D=L+2\).  We obtain a connected finite bipartite cover
\(H_M\to G_M\) and finitely many color classes whose full lifts to the
universal cover are \((L+2)\)-separated.  Passing to a finite cover does not
change the universal cover, so its initial entropy is still \(h_0/M\).

Process the finitely many color classes one at a time.  For every edge in the
active class, first subdivide it once and then subdivide one of the two new
edges once more.  Previous subdivisions only increase distances, and passing
from an active edge to one of its two descendants changes line-graph distance
by at most two.  Thus the lifted active sets are \(L\)-separated at both
single-subdivision steps.  Apply
\cref{lem:separated-subdivision} with \(\Lambda=e^{h_0}\).  Subdivision
never increases entropy, so the same choice of \(\Lambda\) remains valid at
every later stage.  Consequently, each single step changes entropy by at most
\(h_0/(L+2)\).

We record only the stages at which both subdivisions of the active color
class are complete.  At every such stage, each processed edge has
been replaced by a path of length three.  The graph is again bipartite with
the original types, the degrees of old vertices are unchanged, and each
processed edge has created one degree-two vertex of each type.  After every
color has been processed, every unit edge of \(H_M\) has length three, so the
terminal universal-cover entropy is \(h_0/(3M)\).

Every active color class is nonempty.  Therefore each of its two subdivision
steps strictly decreases entropy by
\cref{lem:weighted-subdivision-pressure}; the graph remains a connected
non-cyclic core throughout.  In particular, the completed-stage entropies are
strictly decreasing and strictly positive.

As \(M\) ranges over the positive odd integers, the intervals
\begin{equation}\label{eq:typed-density-intervals}
  \left[\frac{h_0}{3M},\frac{h_0}{M}\right]
\end{equation}
cover \((0,h_0]\): the interval for \(M+2\) overlaps that for \(M\), since
\(1/(M+2)\geq1/(3M)\).  Given \(t\in(0,h_0]\) and \(\varepsilon>0\), choose
such an \(M\), and then choose \(L\) so large that
\(2h_0/(L+2)<\varepsilon\).  Write the initial and completed-stage entropies as
\[
  h^{(0)}=\frac{h_0}{M}>h^{(1)}>\cdots>h^{(N)}
  =\frac{h_0}{3M}.
\]
Two successive terms differ by less than \(\varepsilon\).  If \(t<h^{(0)}\),
let \(k\geq1\) be the least index for which \(h^{(k)}\leq t\).  Then
\(0\leq t-h^{(k)}<\varepsilon\).  If \(t=h^{(0)}\), choose \(h^{(1)}\)
instead; strict entropy decrease and the same successive-stage estimate give
\(0<t-h^{(1)}<\varepsilon\).  Thus in every case the selected completed stage
has entropy at most \(t\) and within \(\varepsilon\) of it; in the endpoint
case \(t=h^{(0)}\), it is strictly below \(t\).  Its graph
contains degree-two vertices of both types.  Mark
one of each as \(p_0,p_1\).  The
graph is finite, connected, non-cyclic, and has minimum degree two; its old
vertices retain the two required degree bounds.  This proves density.

Finally, density lets us choose recursively an exponent strictly between the
previous one and \(\delta\), and within \(1/n\) of \(\delta\).  This gives
\eqref{eq:ported-approximation}.  When \(\delta=h_0\), first choose
targets \(t_n\uparrow h_0\) and then take the density approximants inside
\((t_n,h_0)\); this makes the strict inequality below the endpoint
explicit.
\end{proof}

\subsection{Infinite cores and the unrestricted free spectrum}

Fix an arbitrary bi-infinite binary word
\(u=(u_j)_{j\in\ZZ}\in\{0,1\}^{\ZZ}\).  Put
\(I=\ZZ\setminus\{0\}\), and fix the bijection
\begin{equation}\label{eq:sparse-index-enumeration}
  \nu(j)=
  \begin{cases}
    2j,&j>0,\\
    -2j-1,&j<0,
  \end{cases}
  \qquad I\longrightarrow\{1,2,3,\ldots\}.
\end{equation}
Reindex the sequence in \eqref{eq:ported-approximation} by declaring
\begin{equation}\label{eq:sparse-core-reindexing}
  (C_j,p_{j,0},p_{j,1})
  =(\widehat C_{\nu(j)},\widehat p_{\nu(j),0},
    \widehat p_{\nu(j),1}),
  \qquad j\in I.
\end{equation}
Thus \(\eta(C_j)\to\delta\) as \(\nu(j)\to\infty\), and the rates increase
strictly in the order prescribed by \(\nu\).

We first describe an expanded spine.  Begin with symbolic vertices
\((x_j)_{j\in\ZZ}\), and
replace the symbolic edge \(x_jx_{j+1}\) by a corridor of length \(L_j\),
where
\begin{equation}\label{eq:sparse-corridor-parity}
  L_j\equiv u_j\pmod 2.
\end{equation}
Give \(x_0\) type zero and propagate types along the corridors.  The
positive integers \(L_j\) will be chosen below and can be arbitrarily large
subject to the parity condition.  Write \(\tau_j\in\{0,1\}\) for the type
of \(x_j\).  At every \(x_j\), \(j\neq0\), attach a double-ported typed core
\((C_j,p_{j,0},p_{j,1})\) by the bridge
\(x_jp_{j,1-\tau_j}\).  Thus the bridge joins opposite types; the other
port remains available for degree completion.  For brevity set
\(p_j=p_{j,1-\tau_j}\), and call the resulting bipartite graph
\(Y_{\delta,u}^{r,s}\).

\FloatBarrier
\begin{figure}[H]
\centering
\begin{tikzpicture}[
  x=.98cm,y=1cm,
  spine/.style={line width=1.05pt,blue!68!black},
  corridor/.style={spine,densely dotted},
  bridge/.style={line width=.8pt,gray!70!black},
  block/.style={draw=green!45!black,fill=green!7,rounded corners=2pt,
    minimum width=1.12cm,minimum height=.72cm,line width=.8pt},
  vert/.style={circle,draw=blue!70!black,fill=white,line width=.8pt,
    minimum size=5pt,inner sep=0pt},
  center/.style={vert,fill=orange!25,draw=orange!80!black,
    minimum size=6.3pt}
]
  \node[vert]   (xm2) at (-5.6,0) {};
  \node[vert]   (xm1) at (-2.8,0) {};
  \node[center] (x0)  at (0,0) {};
  \node[vert]   (x1)  at (2.8,0) {};
  \node[vert]   (x2)  at (5.6,0) {};
  \draw[corridor] (xm2)--node[above=2pt,font=\scriptsize]
    {$L_{-2}\equiv u_{-2}\pmod 2$} (xm1);
  \draw[corridor] (xm1)--node[above=2pt,font=\scriptsize]
    {$L_{-1}\equiv u_{-1}\pmod 2$} (x0);
  \draw[corridor] (x0)--node[above=2pt,font=\scriptsize]
    {$L_{0}\equiv u_{0}\pmod 2$} (x1);
  \draw[corridor] (x1)--node[above=2pt,font=\scriptsize]
    {$L_{1}\equiv u_{1}\pmod 2$} (x2);
  \node[block] (Cm2) at (-5.6,-1.25) {$C_{-2}$};
  \node[block] (Cm1) at (-2.8,-1.25) {$C_{-1}$};
  \node[block] (C1)  at ( 2.8,-1.25) {$C_{1}$};
  \node[block] (C2)  at ( 5.6,-1.25) {$C_{2}$};
  \draw[bridge] (xm2)--(Cm2);
  \draw[bridge] (xm1)--(Cm1);
  \draw[bridge] (x1)--(C1);
  \draw[bridge] (x2)--(C2);
  \node[font=\scriptsize,fill=white,inner sep=1.4pt] at (0,-.42) {$x_0$};
  \node[font=\scriptsize,text=green!35!black,anchor=west]
    at (6.3,-1.25) {finite non-cyclic blocks};
  \draw[spine,densely dotted] (-5.5,0)--(xm2);
  \draw[spine,densely dotted] (x2)--(5.5,0);
\end{tikzpicture}
\caption{The sparse typed block quotient before biregular completion.
The orange vertex marks the typed basepoint.  Each finite typed core is joined
at the port opposite to the corresponding spine vertex.  Corridor parities
may be prescribed independently, while their lengths control mixed normal
forms.}
\label{fig:sparse-sparse-blocks}
\end{figure}
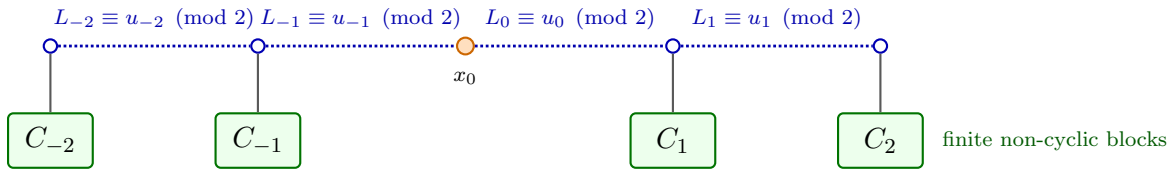

We next choose the corridor lengths.  The reindexing
\eqref{eq:sparse-core-reindexing} fixes which core lies at each marked
vertex.
For \(G_j=\pi_1(C_j,p_j)\), let \(\ell_j(g)\) be the length of the reduced
loop representing \(g\) at \(p_j\), and define
\begin{equation}\label{eq:block-poincare-series}
  \Phi_j(z)=\sum_{1\neq g\in G_j}z^{\ell_j(g)}.
\end{equation}
Its radius of convergence is \(e^{-\eta(C_j)}\).  Hence, with
\(z_0=e^{-\delta}\), every number \(\Phi_j(z_0)\) is finite.

Write, with distances measured in \(Y_{\delta,u}^{r,s}\),
\[
  R_j=d_{Y_{\delta,u}^{r,s}}(x_0,p_j),
  \qquad
  d_{ij}=d_{Y_{\delta,u}^{r,s}}(p_i,p_j).
\]
The first estimate below controls the initial and terminal travel between the
basepoint and a block; the second controls every transition between two
distinct block factors in a free-product normal form.
Since \(0<z_0<1\), choose the corridor lengths inductively by placing
the marked vertices successively in the order
\(x_{-1},x_1,x_{-2},x_2,\ldots\), which is exactly the order determined by
\(\nu\).  When a new vertex \(x_j\) is placed, it is joined to the already
constructed inner part of the spine by one new corridor.  Increasing the
length of that corridor increases \(R_j\) and every distance \(d_{ij}\) to a
previously placed vertex \(x_i\), while leaving all distances between
previously placed vertices unchanged.  Hence its length can be chosen, with
the prescribed parity, so that the new initial condition and both directed
transition conditions involving \(j\) hold simultaneously.  Previously
imposed inequalities remain valid.  Continuing this induction on the two
sides of \(x_0\) gives \eqref{eq:sparse-corridor-parity} and
\begin{align}
  \sum_{j\in I}z_0^{R_j}\Phi_j(z_0)&<\infty,
    \label{eq:sparse-initial-sum}\\
  z_0^{d_{ij}}\Phi_j(z_0)
    &\leq 2^{-\nu(i)-\nu(j)-4}
      \qquad(i\neq j).
    \label{eq:sparse-transfer-smallness}
\end{align}
More explicitly, at the stage when \(x_j\) is placed, require
\(z_0^{R_j}\Phi_j(z_0)\leq2^{-\nu(j)}\).  For every previously placed
\(i\), impose \eqref{eq:sparse-transfer-smallness} both for the ordered pair
\((i,j)\), whose factor is \(\Phi_j(z_0)\), and for \((j,i)\), whose factor
is \(\Phi_i(z_0)\).  These constitute only finitely many lower bounds on the
new corridor length, and an integer of the prescribed parity exists beyond
their maximum.  Summing the initial inequalities gives
\eqref{eq:sparse-initial-sum}, and every ordered pair is handled at a finite
stage, proving \eqref{eq:sparse-transfer-smallness}.

Complete every deficient type-zero, respectively type-one, vertex of
\(Y_{\delta,u}^{r,s}\) by attaching rooted biregular half-trees until its
degree is \(r+1\), respectively \(s+1\); call the result
\(Q_{\delta,u}^{r,s}\).  Before completion, each noncentral marked spine
vertex and each used port has degree three, while every other new corridor
or port vertex has degree two.  Since \(r,s\geq2\), all degrees can be
completed.  The completion adds no reduced closed path.  Thus
\begin{equation}\label{eq:sparse-free-product}
  \Gamma_{\delta,u}^{r,s}:=\pi_1(Q_{\delta,u}^{r,s})
  \cong \mathop{\ast}_{j\in I}G_j,
\end{equation}
and the universal cover of \(Q_{\delta,u}^{r,s}\) is
\(\T_{r+1,s+1}\).

\begin{theorem}[Sparse typed gluing]
\label{thm:sparse-full}
For the graph constructed above,
\[
  \delta_{\Gamma_{\delta,u}^{r,s}}=\delta.
\]
Consequently, sparse block quotients realize every exponent in
\((0,\frac12\log(rs)]\); adjoining the trivial action at exponent zero
gives the full interval \([0,\frac12\log(rs)]\).
\end{theorem}

\begin{proof}
A nontrivial element of the free product
\eqref{eq:sparse-free-product} has a unique normal form
\(g_1\cdots g_m\), where
\(1\neq g_k\in G_{i_k}\) and \(i_k\neq i_{k+1}\).  Because distinct blocks
are joined by a tree of bridges, the length of the corresponding reduced
closed path at \(x_0\) is
\begin{equation}\label{eq:sparse-normal-form-length}
  R_{i_1}+\ell_{i_1}(g_1)
  +\sum_{k=1}^{m-1}
    \bigl(d_{i_k i_{k+1}}+\ell_{i_{k+1}}(g_{k+1})\bigr)
  +R_{i_m}.
\end{equation}
Choose a lift \(\widetilde x_0\) of the quotient base vertex \(x_0\) in the
universal cover.  Under the orbit--loop correspondence,
\eqref{eq:sparse-normal-form-length} is also the displacement of the
corresponding deck transformation at \(\widetilde x_0\).
For \(0<z\leq z_0\), introduce the countable nonnegative matrix
\begin{equation}\label{eq:sparse-transfer-kernel}
  \mathsf K_{\mathrm{tr}}(z)_{ij}
  =\begin{cases}
     z^{d_{ij}}\Phi_j(z),&i\neq j,\\
     0,&i=j.
   \end{cases}
\end{equation}
Thus \(\mathsf K_{\mathrm{tr}}(z)_{ij}\) is the total weight of traveling
from the marked point of
block \(i\) to that of block \(j\) and then choosing a nontrivial syllable in
\(G_j\).
The sum in \eqref{eq:sparse-transfer-smallness} gives
\begin{equation}\label{eq:sparse-kernel-bound}
  \lVert \mathsf K_{\mathrm{tr}}(z_0)\rVert_{\ell^\infty\to\ell^\infty}
  =\sup_i\sum_j\mathsf K_{\mathrm{tr}}(z_0)_{ij}<\frac18.
\end{equation}
Let
\[
  \mathbf w_{\mathrm{in}}(z)
    =\bigl(z^{R_i}\Phi_i(z)\bigr)_{i\in I},
  \qquad
  \mathbf w_{\mathrm{out}}(z)
    =\bigl(z^{R_i}\bigr)_{i\in I}.
\]
Denote by \(\mathcal P_{\Gamma_{\delta,u}^{r,s}}(z)\) the ordinary
displacement growth series with weight
\(z^{d(\widetilde x_0,\gamma\widetilde x_0)}\).  Summing
\eqref{eq:sparse-normal-form-length} over normal forms of all syllable
lengths gives
\begin{equation}\label{eq:sparse-global-poincare}
  \mathcal P_{\Gamma_{\delta,u}^{r,s}}(z)
  =1+\sum_{m\geq1}
    \mathbf w_{\mathrm{in}}(z)^{\mathsf T}
    \mathsf K_{\mathrm{tr}}(z)^{m-1}\mathbf w_{\mathrm{out}}(z).
\end{equation}
Since all coefficients are nonnegative, the rearrangement of the sums over syllable lengths, block indices, and normal-form length is justified by Tonelli’s theorem.
At \(z=z_0\), the vector \(\mathbf w_{\mathrm{in}}(z_0)\) is in
\(\ell^1(I)\) by \eqref{eq:sparse-initial-sum}, while
\(\mathbf w_{\mathrm{out}}(z_0)\) is bounded by one.
For every \(m\geq1\), the \(\ell^1\)--\(\ell^\infty\) pairing gives
\begin{align*}
  \left|
  \mathbf w_{\mathrm{in}}(z_0)^{\mathsf T}
  \mathsf K_{\mathrm{tr}}(z_0)^{m-1}
  \mathbf w_{\mathrm{out}}(z_0)
  \right|
  &\leq
  \|\mathbf w_{\mathrm{in}}(z_0)\|_{\ell^1}
  \|\mathsf K_{\mathrm{tr}}(z_0)\|_{\ell^\infty\to\ell^\infty}^{m-1}
  \|\mathbf w_{\mathrm{out}}(z_0)\|_{\ell^\infty}.
\end{align*}
The Neumann series in \eqref{eq:sparse-global-poincare} therefore converges
by \eqref{eq:sparse-kernel-bound}.  It continues to converge for
\(0<z<z_0\) by coefficientwise monotonicity.  Hence
\(\delta_{\Gamma_{\delta,u}^{r,s}}\leq\delta\).

If \(0\leq t<\delta\), choose \(j\) with \(\eta(C_j)>t\).  Then
\(e^{-t}>e^{-\eta(C_j)}\), so \(\Phi_j(e^{-t})\) diverges.  Its contribution
to the series based at \(x_0\) is
\[
  e^{-2tR_j}\Phi_j(e^{-t})=+\infty.
\]
Hence the Poincar\'e series of the subgroup
\(G_j<\Gamma_{\delta,u}^{r,s}\), and therefore that of the full group,
diverges at \(t\).  Consequently
\(\delta_{\Gamma_{\delta,u}^{r,s}}\geq\delta\).  This proves equality.
The action is free, and hence discrete, because it is the deck action of the
ordinary graph covering
\(\T_{r+1,s+1}\to Q_{\delta,u}^{r,s}\).  Its quotient is bipartite, so the
action is type-preserving.  For \(\delta=0\), use the trivial group.
\end{proof}

\begin{proof}[Proof of \cref{thm:intro-stratification}\textup{(i)}]
Apply \cref{lem:ported-density,thm:sparse-full}.  The reindexing
\eqref{eq:sparse-core-reindexing} makes the block rates increase strictly to
\(\delta\) in both directions along the spine.  Every block is non-cyclic by
the definition of a double-ported typed core.
\end{proof}

\begin{remark}[Parity freedom in the marked construction]
The corridor parity condition \eqref{eq:sparse-corridor-parity} may be
prescribed by any bi-infinite binary word \(u\).  When the quotient is
regarded together with its based, oriented spine, its corridor parities
recover \(u\).  Retaining only the oriented spine identifies words up to
shift, and forgetting the orientation also allows reflection.  We make no
claim that the spine is intrinsically unique in the unmarked quotient; this
parity freedom is a refinement of the construction, not an additional
classification result.
\end{remark}

\subsection{Finite generation and the finite-state spectrum}
\label{sec:fg-spectrum}

We now prove the typed description \eqref{eq:Ars-intro}; the regular
statement will be derived afterward.

\begin{theorem}[Finite-core realization]\label{thm:finite-core-realization}
Let \(C=C_0\sqcup C_1\) be a finite connected bipartite core satisfying
\[
  \deg_C(v)\leq r+1\quad(v\in C_0),
  \qquad
  \deg_C(v)\leq s+1\quad(v\in C_1).
\]
There is a connected bipartite \((r+1,s+1)\)-biregular graph \(Q_C\),
obtained from \(C\) by attaching trees, such that
\[
  \Core(Q_C)=C,
  \qquad
  \widetilde{Q_C}\cong\T_{r+1,s+1}.
\]
Its deck group \(\Gamma_C=\pi_1(Q_C)\) is finitely generated and free, and
\[
  \delta_{\Gamma_C}=\log\rho(B_C).
\]
\end{theorem}

\begin{proof}
At each deficient vertex attach the required number of rooted biregular
half-trees, with the opposite type at the first new vertex.  A reduced closed
path cannot enter an attached tree: at a vertex of maximal distance from the
attachment point it would have to reverse its last edge.  Hence the core and
fundamental group remain unchanged.  The universal cover has the prescribed
typed degrees and is therefore \(\T_{r+1,s+1}\).  Finally apply
\cref{prop:finite-core-growth}.
\end{proof}

\begin{proof}[Proof of \cref{thm:intro-stratification}\textup{(ii)--(iii)}]
Let \(\Gamma\) be a finitely generated free group acting freely on
\(\T_{r+1,s+1}\) type-preservingly.  If \(\Gamma\) is trivial, its
multiplicative growth rate \(\lambda_\Gamma=e^{\delta_\Gamma}\) is \(1\).  The same spectral value is also
realized, by a different action, by the infinite cyclic deck group of an even
cycle.  If \(\Gamma\) is nontrivial, the quotient is bipartite and has, by
\cref{lem:finite-core}, a
finite core \(C=\Core(\Gamma\backslash\T_{r+1,s+1})=C_0\sqcup C_1\).  Its type-zero and type-one degrees
lie in \([2,r+1]\) and \([2,s+1]\), respectively, and
\(\exp(\delta_\Gamma)=\rho(B_C)\).  Conversely, every core with these
properties is realized by \cref{thm:finite-core-realization}.  This proves
\eqref{eq:Ars-intro}.

There are only countably many finite graphs, so \(\cA_{r,s}\) is countable,
and \eqref{eq:biregular-rho-bound} gives
\(\cA_{r,s}\subseteq[1,\sqrt{rs}]\).  The density assertion follows from
\cref{lem:ported-density}, since every double-ported typed core is one of the
cores allowed in \eqref{eq:Ars-intro}; the value \(1\) is supplied by an
even cycle.
\end{proof}

\begin{corollary}[Regular finite-core spectrum]
\label{cor:regular-finite-core-spectrum}
For \(q\geq2\), let \(\cA_q\) be the exponential critical-exponent spectrum
of finitely generated free groups acting freely on \(\T_{q+1}\), without a
type-preserving hypothesis.  Then
\[
  \cA_q=\cA_{q,q}
  =\{\rho(B_C):C\text{ is a finite connected core and }
      2\leq\deg_C(v)\leq q+1\}.
\]
In particular, \(\cA_q\) is countable and dense in \([1,q]\).
\end{corollary}

\begin{proof}
The untyped finite-core description follows from the same tree-completion
argument.  If such a core \(C\) is not bipartite, replace it by its canonical
connected bipartite double cover \(C^{\mathrm{bip}}\).  The two graphs have
the same universal cover, and the growth of reduced paths in that universal
cover gives
\(\rho(B_{C^{\mathrm{bip}}})=\rho(B_C)\).  Hence every untyped value has a
type-preserving realization.  The reverse inclusion is immediate, and the
remaining assertions follow from \cref{thm:intro-stratification} with
\(r=s=q\).
\end{proof}

\begin{proof}[Proof of \cref{cor:intro-regular}]
Set \(r=s=q\).  The ambient entropy bound gives
\(\delta_\Gamma\leq\log q\) for every discrete subgroup of
\(\Aut(\T_{q+1})\), while \cref{thm:intro-stratification}\textup{(i)}
realizes every value in \([0,\log q]\) by a free type-preserving action.
The equality \(\cA_{q,q}=\cA_q\) is
\cref{cor:regular-finite-core-spectrum}.
\end{proof}

\subsection{The Ihara--Bass determinant and algebraic restrictions}
\label{sec:ihara}

Let \(C\) be a finite connected core with \(n\) vertices and \(m\) unoriented
edges.  Let \(A_C\) be its adjacency matrix, let \(D_C\) be its diagonal
degree matrix, and let \(I_j\) denote the \(j\times j\) identity matrix.
For loops and multiple edges we use the half-edge convention: a
loop contributes two to the degree and two to the corresponding diagonal
entry of \(A_C\), and its two orientations are distinct states of \(B_C\).
Let \(v\) be a formal variable.  A closed non-backtracking path is
\emph{tailless} when its last oriented edge is not the reverse of its first,
and it is \emph{primitive} when it is not a proper power of a shorter closed
path.  Thus a tailless non-backtracking closed path is cyclically reduced in
the terminology of \cref{sec:dictionary}.  If \(\mathfrak p\) is primitive,
let \(\ell(\mathfrak p)\) be its length and let \([\mathfrak p]\) denote its
equivalence class under cyclic change of starting edge; the reversed path is
not identified with \(\mathfrak p\).  The Ihara zeta function is
\[
  Z_C(v)=\prod_{[\mathfrak p]}(1-v^{\ell(\mathfrak p)})^{-1},
\]
where the product ranges over all such equivalence classes.
Hashimoto's edge determinant and the Ihara--Bass formula give
\begin{align}
  Z_C(v)^{-1}
  &=\det(I_{2m}-vB_C),\label{eq:hashimoto-det}\\
  &=(1-v^2)^{m-n}
     \det\!\left(I_n-vA_C+v^2(D_C-I_n)\right).
     \label{eq:ihara-bass}
\end{align}
See \cite{Hashimoto1989,Bass1992,KotaniSunada2000}.  Equivalently,
\begin{equation}\label{eq:char-bass}
  \det(xI_{2m}-B_C)
  =(x^2-1)^{m-n}
    \det\!\left(x^2I_n-xA_C+(D_C-I_n)\right).
\end{equation}

Define the vertex form of the Ihara--Bass polynomial by
\begin{equation}\label{eq:PC}
  P_C(x)=\det\!\left(x^2I_n-xA_C+(D_C-I_n)\right)\in\ZZ[x].
\end{equation}
For \(\rho(B_C)>1\), the number \(\rho(B_C)\) is the largest positive zero
selected by the non-backtracking Perron root, and its minimal polynomial
divides \(P_C(x)\).

\begin{definition}
A real algebraic integer \(\lambda>1\) is a \emph{Perron number} if every
other Galois conjugate \(\lambda'\) satisfies
\(|\lambda'|<\lambda\).  It is a \emph{weak Perron number} if the strict
inequality is replaced by \(|\lambda'|\leq\lambda\).
An irreducible nonnegative matrix is \emph{primitive} if one of its positive
powers has all entries positive.  Its \emph{period} is the greatest common
divisor of the lengths of the directed cycles in the graph having an arrow
\(i\to j\) exactly when the \((i,j)\)-entry is positive.
\end{definition}

\begin{proposition}[Arithmetic restrictions]\label{prop:arithmetic}
Let \(C\) be a finite connected non-cyclic core, and put
\(\lambda=\rho(B_C)>1\).  Then:
\begin{enumerate}[label=(\roman*)]
  \item \(\lambda\) is a weak Perron algebraic integer;
  \item if \(B_C\) is primitive, then \(\lambda\) is a Perron number;
  \item the degree of the minimal polynomial of \(\lambda\) is at most
    \(2|V(C)|\);
  \item if \(m_\lambda(x)\) is the monic minimal polynomial, then
    \[
      |m_\lambda(0)|
      \mathrel{\big|}
      \prod_{v\in V(C)}(\deg_C(v)-1).
    \]
\end{enumerate}
\end{proposition}

\begin{proof}
The matrix \(B_C\) is irreducible, nonnegative, and integral.  Its spectral
radius is therefore a positive algebraic integer, and every Galois conjugate
of \(\lambda\) is also an eigenvalue of \(B_C\); hence its modulus is at most
\(\lambda\).  If \(B_C\) is primitive, Perron--Frobenius theory makes this
inequality strict for every other eigenvalue.

For \(\lambda>1\), \eqref{eq:char-bass} shows that
\(m_\lambda(x)\) divides the degree-\(2|V(C)|\) polynomial \(P_C(x)\).  By
Gauss' lemma the quotient is monic and integral.  Since
\[
  P_C(0)=\det(D_C-I_n)=\prod_v(\deg_C(v)-1),
\]
the divisibility of constant terms follows.
\end{proof}

The use of weak Perron numbers is essential.  Subdivision may introduce
periodicity into \(B_C\), producing conjugates of equal modulus.  More
generally, if \(p=p(B_C)\) is the period just defined, then \(\lambda^p\) is a
Perron number.  Indeed, every peripheral eigenvalue has the form
\(\lambda\zeta\) with \(\zeta^p=1\); after taking \(p\)-th powers these values
coalesce.  To account for every Galois conjugate, let
\(\sigma:\mathbb Q(\lambda^p)\hookrightarrow\mathbb C\) be an embedding and
extend it to an embedding of \(\mathbb Q(\lambda)\).  Then
\(\sigma(\lambda^p)=\sigma(\lambda)^p\), and \(\sigma(\lambda)\) is a Galois
conjugate of \(\lambda\), hence an eigenvalue of \(B_C\).  If it is
peripheral, its \(p\)-th power equals \(\lambda^p\); otherwise its modulus is
strictly smaller than \(\lambda^p\).  Thus every conjugate of \(\lambda^p\)
other than \(\lambda^p\) itself has strictly smaller modulus.

\subsection{Bipartite symmetry and transcendence}

There is also a useful symmetry in the type-preserving case.

\begin{proposition}[Bipartite symmetry]\label{prop:bipartite-symmetry}
If \(C\) has \(m\) unoriented edges and
\(V(C)=C_0\sqcup C_1\) is a bipartition, then, after ordering oriented edges
by the type of their terminal vertex, there are \(m\times m\) zero-one
matrices \(M\) and \(N\) such that
\[
  B_C=\begin{pmatrix}0&M\\N&0\end{pmatrix}.
\]
Consequently, \(\det(xI_{2m}-B_C)\) and \(P_C(x)\) are even polynomials, and
\[
  \rho(B_C)^2=\rho(MN)=\rho(NM).
\]
For a finitely generated type-preserving action on
\(\T_{r+1,s+1}\), one may therefore write
\[
  \delta_\Gamma
  =\frac12\log\rho(MN).
\]
\end{proposition}

\begin{proof}
A non-backtracking step changes the type of the terminal vertex, giving the
displayed block form.  Squaring the matrix gives diagonal blocks \(MN\) and
\(NM\).  For the polynomial assertion, let \(S\) be the diagonal matrix
whose entries are \(+1\) on \(C_0\) and \(-1\) on \(C_1\).  Then
\(S A_C S=-A_C\), while
\(S(D_C-I_{|V(C)|})S=D_C-I_{|V(C)|}\), which proves the evenness.
\end{proof}

Since \(\exp(\delta_\Gamma)\) is algebraic for a finitely generated free
action, Lindemann--Weierstrass gives a complementary statement about the
critical exponent itself; see, for example, \cite{Baker1990}.

\begin{corollary}\label{cor:transcendental-delta}
If a finitely generated free action has \(\delta_\Gamma>0\), then
\(\delta_\Gamma\) is transcendental.  Thus its algebraic structure is most
naturally recorded by the algebraic integer
\(\lambda_\Gamma=\exp(\delta_\Gamma)\), not by \(\delta_\Gamma\) itself.
\end{corollary}

\begin{proof}
If \(\delta_\Gamma\neq0\) were algebraic, the Lindemann--Weierstrass theorem
would make \(\exp(\delta_\Gamma)\) transcendental, contrary to
\cref{prop:arithmetic}.
\end{proof}

% =============================================================================
\section{Fixed-rank entropy strata and inverse realization}
\label{sec:inverse}
% =============================================================================

\subsection{Kernel reduction and weighted turn determinants}

Let \(C\) be a finite connected non-cyclic core.  Suppress every maximal
path whose internal vertices have degree two.  The resulting connected
multigraph \(K\), in which loops and parallel edges are retained, is called
the \emph{kernel} of \(C\).  Every vertex of \(K\) has degree at least three.
An unoriented edge \(e\in E(K)\) records a path in \(C\); denote its positive
integral length by \(\ell(e)\), and set
\(\ell(\bar e)=\ell(e)\) for the reverse orientation.  We write
\(C=K(\ell)\) for the subdivision reconstructed from this data.
The convention \(\mathbb N=\{1,2,3,\ldots\}\) fixed in the
introduction is used here.

The \emph{turn matrix} \(T_K\), indexed by \(\vec E(K)\), is defined by
\[
  (T_K)_{e,f}
  =
  \begin{cases}
    1,&t(e)=o(f)\text{ and }f\neq\bar e,\\
    0,&\text{otherwise}.
  \end{cases}
\]
For \(x>1\), put
\[
  D_\ell(x)
  =\operatorname{diag}\bigl(x^{\ell(f)}:f\in\vec E(K)\bigr)
\]
and define the weighted turn matrix
\begin{equation}\label{eq:weighted-turn}
  M_{K,\ell}(x)=T_KD_\ell(x)^{-1}.
\end{equation}
Thus an allowed transition whose next kernel edge is \(f\) has weight
\(x^{-\ell(f)}\).  All notation in the following theorem has now been
specified.

\begin{theorem}[Kernel pressure equation]\label{thm:kernel-pressure}
Let \(C=K(\ell)\) be a finite connected non-cyclic core.  There is a unique
number \(\alpha_{K,\ell}>1\) such that
\begin{equation}\label{eq:kernel-pressure-defined}
  \rho\!\left(M_{K,\ell}(\alpha_{K,\ell})\right)=1,
\end{equation}
and this number is
\[
  \alpha_{K,\ell}=\rho(B_C).
\]
Equivalently, for every prescribed \(\alpha>1\),
\begin{equation}\label{eq:kernel-pressure}
  \rho(B_C)=\alpha
  \quad\Longleftrightarrow\quad
  \rho\!\left(M_{K,\ell}(\alpha)\right)=1.
\end{equation}
The function \(x\mapsto\rho(M_{K,\ell}(x))\) is continuous and strictly
decreasing from a value greater than one to zero.

The pressure-selected number \(\alpha_{K,\ell}\) is the distinguished
positive zero of the monic integral polynomial
\begin{equation}\label{eq:kernel-polynomial}
  F_{K,\ell}(x)
  =\det\!\left(D_\ell(x)-T_K\right)\in\ZZ[x].
\end{equation}
Here ``distinguished'' means that
\(\rho(M_{K,\ell}(\alpha_{K,\ell}))=1\); no uniqueness assertion is made
about positive zeros of the determinant that are not selected by the
spectral-radius equation.  In particular, the minimal polynomial of
\(\alpha_{K,\ell}\) divides
\(F_{K,\ell}(x)\).
\end{theorem}

\begin{proof}
The turn matrix \(T_K\) is precisely the Hashimoto matrix \(B_K\), and
\(M_{K,\ell}(x)=W_{K,\ell}(x)\) in the notation of
\eqref{eq:early-weighted-turn}.  Since \(C=K(\ell)\), the pressure equation,
continuity, strict monotonicity, and uniqueness are exactly
\cref{lem:weighted-subdivision-pressure}.

Finally,
\[
  \det(I-T_KD_\ell(x)^{-1})
  =\frac{\det(D_\ell(x)-T_K)}{\det D_\ell(x)}.
\]
The leading term of the numerator is
\(\prod_{f\in\vec E(K)}x^{\ell(f)}\), with coefficient one, and every other
term has smaller degree.  This proves the polynomial and divisibility
assertions.
\end{proof}

Give each vertex of a kernel \(K\) a type through a map
\(\tau:V(K)\to\mathbb Z/2\mathbb Z\).  A length vector
\(\ell\in\mathbb N^{E(K)}\) is \emph{admissible for \(\tau\)} if
\begin{equation}\label{eq:kernel-parity}
  \ell(e)\equiv\tau(o(e))+\tau(t(e))\pmod2
\end{equation}
for every oriented kernel edge \(e\).  This is equivalent to the existence
of a bipartite type extension over the subdivision \(K(\ell)\).

Two typed kernels \((K,\tau)\) and \((K',\tau')\) are
\emph{type-preservingly isomorphic} if there is a multigraph isomorphism
\(\varphi:K\to K'\) such that \(\tau'\circ\varphi=\tau\).  For kernel--length
data we additionally require
\(\ell'(\varphi(e))=\ell(e)\) for every unoriented edge \(e\in E(K)\).

\begin{theorem}[Typed finite-kernel parametrization]
\label{thm:fixed-rank-kernels}
Fix \(g\geq2\) and \(r,s\geq2\).  Up to type-preserving isomorphism, let
\(\mathcal K_{g;r,s}\) be the finite set of pairs \((K,\tau)\) satisfying
\[
  b_1(K)=g,
  \qquad
  3\leq\deg_K(v)\leq
  \begin{cases}
    r+1,&\tau(v)=0,\\
    s+1,&\tau(v)=1.
  \end{cases}
\]
Then the Hashimoto spectral radii of finite connected bipartite cores of
rank \(g\) satisfying the two biregular degree bounds are exactly
\begin{equation}\label{eq:fixed-rank-parametrization}
  \bigcup_{(K,\tau)\in\mathcal K_{g;r,s}}
  \left\{
    \alpha_{K,\ell}:\ell\in\mathbb N^{E(K)}
      \text{ is admissible for }\tau
  \right\},
\end{equation}
where \(\alpha_{K,\ell}>1\) is the unique solution of
\(\rho(M_{K,\ell}(\alpha_{K,\ell}))=1\).  Moreover,
\[
  |V(K)|\leq2g-2,
  \qquad
  |E(K)|\leq3g-3.
\]
\end{theorem}

\begin{proof}
Kernel suppression preserves the first Betti number and the types and
degrees of all branching vertices.  It forces \eqref{eq:kernel-parity}
because the type changes across every unit edge.  Conversely, an admissible
length vector extends the kernel types across every subdivided edge and
reconstructs a bipartite core with the required degree bounds.  The exact
spectral parametrization follows from \cref{thm:kernel-pressure}.

Since \(b_1(K)=|E(K)|-|V(K)|+1=g\), the handshaking identity gives
\[
  \sum_{v\in V(K)}(\deg_K(v)-2)=2g-2.
\]
Every summand is at least one.  Hence \(|V(K)|\leq2g-2\), and then
\(|E(K)|=|V(K)|+g-1\leq3g-3\).  Only finitely many multigraphs, including
loops and parallel edges, have these bounds, and each has finitely many type
maps.
\end{proof}

\subsection{Finiteness for a prescribed exponential rate}

The finite-kernel theorem parametrizes all fixed-rank values, but a priori a
single value might have infinitely many length realizations.  The next result
rules this out.

\begin{lemma}[Dickson]\label{lem:dickson}
Every antichain in \(\mathbb N^m\), ordered coordinatewise, is finite.
\end{lemma}

\begin{proof}
Equivalently, every infinite sequence in \(\mathbb N^m\) contains two terms
\(u_i\leq u_j\) with \(i<j\).  This follows by induction on \(m\).  For the
induction step, either some first coordinate occurs infinitely often, in
which case one applies the induction hypothesis to the remaining
coordinates, or one passes to a subsequence with strictly increasing first
coordinates and again applies the induction hypothesis to the remaining
coordinates.
\end{proof}

\begin{theorem}[Fixed-rank finiteness]\label{thm:fixed-rank-finiteness}
Fix \(\alpha>1\), \(g\geq2\), and \(r,s\geq2\).  There are only finitely
many typed kernel--length data \((K,\tau,\ell)\), up to type-preserving
isomorphism, for which
\[
  (K,\tau)\in\mathcal K_{g;r,s},
  \qquad \ell\text{ is admissible for }\tau,
  \qquad
  \rho(B_{K(\ell)})=\alpha.
\]
Consequently, only finitely many type-preserving isomorphism classes of
rank-\(g\) typed core subdivisions realize \(\alpha\).
\end{theorem}

\begin{proof}
There are finitely many possible typed kernels by
\cref{thm:fixed-rank-kernels}, so fix one of them.  If
\(\ell'\geq\ell\) coordinatewise and the inequality is strict in at least one
coordinate, then
\[
  M_{K,\ell'}(\alpha)
  \leq M_{K,\ell}(\alpha)
\]
with a strict inequality in at least one allowed entry.  Both matrices are
irreducible and have the same support.  Strict Perron--Frobenius comparison
therefore gives
\[
  \rho(M_{K,\ell'}(\alpha))
  <\rho(M_{K,\ell}(\alpha)).
\]
Thus the admissible length vectors realizing \(\alpha\) form an antichain
in \(\mathbb N^{E(K)}\), which is finite by \cref{lem:dickson}.
\end{proof}

\subsection{Degeneration and rank drop}

Recall that for \(g\geq2\) and \(r,s\geq2\) we write
\[
  \mathcal R_{g;r,s}
  =
  \left\{\rho(B_C)>1:
  \begin{array}{l}
    C=C_0\sqcup C_1\text{ is a finite connected bipartite core},\\
    b_1(C)=g,\quad
    \deg_C(v)\leq r+1\ (v\in C_0),\\
    \deg_C(v)\leq s+1\ (v\in C_1)
  \end{array}\right\}.
\]
The kernel description also controls how a sequence of fixed-rank algebraic
rates can accumulate.

For a nonnegative matrix, its \emph{support digraph} has one vertex for each
state and an arrow for each positive entry.  A state is \emph{recurrent} if it
lies in a strongly connected component containing a directed cycle; all other
states are \emph{transient}.

\begin{lemma}[Recurrent blocks after deleting weighted edges]
\label{lem:recurrent-bounded-core}
Let \(K\) be a finite connected kernel, let \(E_{\mathrm b}\subseteq E(K)\),
where the subscript \({\mathrm b}\) anticipates the bounded-edge set in
the application below, and assign positive transition weights to the orientations of edges in
\(E_{\mathrm b}\) and weight zero to all remaining columns of the turn
matrix.  After zero and transient states are removed, the irreducible diagonal
blocks come from the connected components of the cyclic core of the subgraph
\(K_{\mathrm b}\) spanned by \(E_{\mathrm b}\).  Each non-cyclic component
contributes one irreducible weighted non-backtracking block; a cycle
contributes two deterministic blocks, one for each orientation.
\end{lemma}

\begin{proof}
A recurrent state belongs to a directed cycle in the support digraph.  Since
every column outside \(E_{\mathrm b}\) is zero, every state on such a cycle is
an orientation of an edge in \(E_{\mathrm b}\).  Directed cycles in the turn
graph are exactly cyclically reduced closed edge paths in
\(K_{\mathrm b}\).  Their union is the oriented edge set of the cyclic core.
Within each connected cyclic-core component, the non-backtracking states form
one irreducible class unless the component is a cycle, in which case its two
orientations give the two deterministic irreducible classes.  Distinct
cyclic-core components cannot communicate recurrently.  This gives exactly
the asserted block decomposition, with the inherited positive weights.
\end{proof}

\begin{lemma}[Pressure invariance under suppression]
\label{lem:pressure-suppression}
Let \(H\) be a finite connected non-cyclic core with a positive integral
length \(a(e)\) on every edge.  Suppress all maximal degree-two paths to
obtain a kernel \(K'\), and give the resulting kernel edge the sum
\(\ell'(e')\) of the lengths along the suppressed path.  Then, for every
\(x>1\),
\[
  \rho\!\left(W_{H,a}(x)\right)=1
  \quad\Longleftrightarrow\quad
  \rho\!\left(M_{K',\ell'}(x)\right)=1.
\]
\end{lemma}

\begin{proof}
The subdivided graphs \(H(a)\) and \(K'(\ell')\) are canonically
isomorphic: suppression merely combines each forced degree-two chain, and
its new length is the sum of the old lengths.  By
\cref{lem:weighted-subdivision-pressure,thm:kernel-pressure}, the unique
value-one point of either weighted operator is the Hashimoto spectral
radius of this same subdivided graph.  The two displayed pressure equations
are therefore equivalent.
\end{proof}

\begin{lemma}[Proper-core Betti drop]
\label{lem:proper-core-betti-drop}
Let \(K\) be a finite connected multigraph of minimum degree at least three.
If \(H\) is a connected core contained in a proper edge subgraph of \(K\),
then
\[
  b_1(H)<b_1(K).
\]
\end{lemma}

\begin{proof}
The inclusion of graphs injects the cycle space of \(H\) into that of \(K\).
Suppose the dimensions were equal.  Then every cycle-space vector of \(K\)
would be supported on \(E(H)\), so no edge of \(K\setminus H\) could belong to
a cycle of \(K\).  Equivalently, every such edge would be a bridge.  Contract
the connected subgraph \(H\) to one vertex.  The remaining nonempty finite
graph is then a forest, because any cycle in the contraction would lift to a
cycle of \(K\) containing an edge outside \(H\).  Choose a leaf of this forest
different from the contracted vertex.  Its preimage is a vertex of \(K\)
outside \(H\) of degree one, contrary to the minimum-degree assumption.
Hence the inclusion of cycle spaces is proper, and the displayed inequality
follows.
\end{proof}

\begin{theorem}[Rank drop at an accumulation point]
\label{thm:rank-drop}
Let \(\alpha_n\in\mathcal R_{g;r,s}\) be pairwise distinct and suppose that
\(\alpha_n\to\alpha>1\).  Then
\[
  \alpha\in\mathcal R_{g_0;r,s}
\]
for some \(2\leq g_0<g\).  In particular,
\(\mathcal R_{2;r,s}\) has no accumulation point in \((1,\infty)\), and its
only accumulation point in \([1,\infty)\) is \(1\).
\end{theorem}

\begin{proof}
Represent \(\alpha_n\) by typed kernel--length data
\((K_n,\tau_n,\ell_n)\).  After passing to a subsequence, the finite-kernel
theorem allows us to assume \((K_n,\tau_n)=(K,\tau)\).  A diagonal
subsequence further ensures that,
for every \(e\in E(K)\), the integer sequence \(\ell_n(e)\) is either
constant or tends to infinity.  At least one edge length tends to infinity,
for otherwise only finitely many length vectors would occur and the
\(\alpha_n\) could not be pairwise distinct.

For every \(n\),
\[
  \rho(M_{K,\ell_n}(\alpha_n))=1.
\]
The matrices converge to a nonnegative matrix \(M_\infty\): the columns
belonging to unbounded edges vanish, while the other columns retain their
fixed lengths with \(x=\alpha\).  Continuity of spectral radius gives
\(\rho(M_\infty)=1\).  Let \(K_{\mathrm b}\), where the subscript
\({\mathrm b}\) stands for ``bounded'', be the subgraph of \(K\) formed by
the bounded edges, equipped with their eventual constant lengths.  After zero
and transient states are removed,
\cref{lem:recurrent-bounded-core} identifies the diagonal irreducible blocks
of \(M_\infty\) with the irreducible components of the weighted
non-backtracking operators of the connected components of the cyclic core of
\(K_{\mathrm b}\).  After the strongly connected components are ordered
topologically, \(M_\infty\) is block upper triangular.  Its spectral radius
is therefore the maximum of the spectral radii of its recurrent diagonal
blocks; the zero and transient diagonal blocks have spectral radius zero.
Since \(\rho(M_\infty)=1\), at least one recurrent block has spectral radius
one.  It cannot arise from a cycle, because every transition weight on a
weighted cycle is strictly less than one, so the spectral radius of each
deterministic cycle block is less than one.

Choose a non-cyclic connected component \(H\) of the cyclic core of
\(K_{\mathrm b}\) whose weighted non-backtracking block has spectral radius
one.  The component \(H\) inherits its vertex types, degree bounds, and fixed
positive integral edge lengths from \((K,\tau,\ell_n)\).  Suppressing its
degree-two vertices produces a typed kernel \((K',\tau')\) and an admissible
length vector \(\ell'\).  By \cref{lem:pressure-suppression}, forced-chain
suppression preserves the value-one pressure equation.  Hence
\[
  \rho(M_{K',\ell'}(\alpha))=1.
\]
Thus \(\alpha=\rho(B_{K'(\ell')})\), and \(K'(\ell')\) is a finite connected
bipartite core satisfying the two biregular degree bounds.

The core \(H\) lies in the proper edge subgraph \(K_{\mathrm b}\), because at
least one edge of \(K\) was discarded.  By
\cref{lem:proper-core-betti-drop}, \(b_1(H)<g\).  Suppression preserves the
first Betti number, so \(g_0=b_1(K')<g\).  Since
\(\alpha>1\), the new core is non-cyclic and \(g_0\geq2\).

For \(g=2\), no such \(g_0\) exists, so there is no accumulation above one.
On the other hand, the bipartite theta graphs with three paths of common
length \(L\) satisfy both typed degree bounds and have spectral radius
\(2^{1/L}\), which tends to one.
\end{proof}

\begin{proof}[Proof of \cref{thm:intro-fixed-rank}]
Part \textup{(i)} is \cref{thm:fixed-rank-kernels}, part \textup{(ii)} is
\cref{thm:fixed-rank-finiteness}, and part \textup{(iii)} is
\cref{thm:rank-drop}.
\end{proof}

\subsection{Complete inverse realization at rank two}

When \(g=2\), the excess-degree identity in the proof of
\cref{thm:fixed-rank-kernels} reads
\[
  \sum_v(\deg_K(v)-2)=2.
\]
It leaves precisely three kernels: one vertex of degree four carrying two
loops; two degree-three vertices joined by three parallel edges; and two
degree-three vertices, each carrying one loop, joined by one edge.  Their
subdivisions are respectively called a \emph{figure-eight},
\emph{theta graph}, and \emph{dumbbell graph}.

\begin{figure}[H]
\centering
\begin{tikzpicture}[
  x=1cm,y=1cm,
  edge/.style={line width=1.05pt,blue!62!black},
  bridge/.style={line width=1.1pt,gray!70!black},
  branch/.style={circle,draw=blue!70!black,fill=blue!58,
    inner sep=2.1pt},
  lab/.style={font=\small,fill=white,inner sep=1.5pt}
]
  % Figure-eight kernel.
  \node[branch] (f) at (0,0) {};
  \draw[edge] (f) .. controls (-1.25,.95) and (-1.25,-.95) .. (f);
  \draw[edge] (f) .. controls (1.25,.95) and (1.25,-.95) .. (f);
  \node[lab] at (-1.03,0) {$a$};
  \node[lab] at (1.03,0) {$b$};
  \node[font=\small\bfseries] at (0,-1.18) {figure-eight};

  % Theta kernel.
  \node[branch] (tu) at (4,0) {};
  \node[branch] (tv) at (6,0) {};
  \draw[edge] (tu) .. controls (4.65,.9) and (5.35,.9) .. (tv);
  \draw[edge] (tu)--(tv);
  \draw[edge] (tu) .. controls (4.65,-.9) and (5.35,-.9) .. (tv);
  \node[lab] at (5,.65) {$a$};
  \node[lab] at (5,.12) {$b$};
  \node[lab] at (5,-.65) {$c$};
  \node[font=\small\bfseries] at (5,-1.18) {theta};

  % Dumbbell kernel.
  \node[branch] (du) at (9.25,0) {};
  \node[branch] (dv) at (11.75,0) {};
  \draw[edge] (du) .. controls (8.15,.83) and (8.15,-.83) .. (du);
  \draw[bridge] (du)--(dv);
  \draw[edge] (dv) .. controls (12.85,.83) and (12.85,-.83) .. (dv);
  \node[lab] at (8.38,0) {$a$};
  \node[lab] at (10.5,.18) {$c$};
  \node[lab] at (12.62,0) {$b$};
  \node[font=\small\bfseries] at (10.5,-1.18) {dumbbell};
\end{tikzpicture}
\caption{The three rank-two kernels.  A label denotes the positive integral
length of the corresponding subdivided kernel edge.}
\label{fig:rank-two-kernels}
\end{figure}
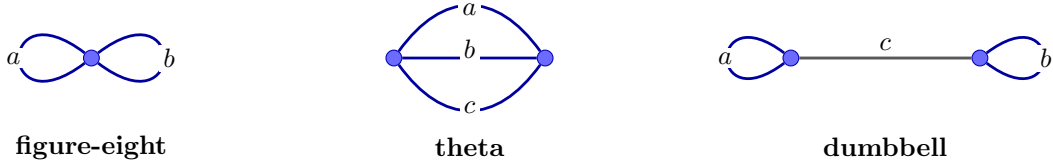

Explicit Ihara zeta functions for all finite connected graphs without
degree-one vertices and circuit rank two were computed by Kwon--Lee
\cite{KwonLee2020}; the zeta function was subsequently shown to be a complete
invariant at rank two \cite{ChicoMattmanRichards2025}.  We use the kernel
pressure equation to select the Perron root, prove uniqueness, and formulate
the result as an inverse realization theorem.

\begin{theorem}[Rank-two polynomial classification]
\label{thm:rank-two-polynomials}
Let \(C\) be a finite connected core with \(b_1(C)=2\), and put
\(\alpha=\rho(B_C)\).  The kernel of \(C\) is of exactly one of the following
three types.
\begin{enumerate}[label=\textup{(\roman*)}]
  \item If \(C\) is a figure-eight with cycle lengths \(a,b\geq1\), then
  \begin{align}
    \frac{2}{\alpha^a+1}+\frac{2}{\alpha^b+1}&=1,
      \label{eq:figure-eight-pressure}\\
    (\alpha^a-1)(\alpha^b-1)&=4,                \label{eq:figure-eight-product}
  \end{align}
  and \(\alpha\) is the unique root \(x>1\) of
  \begin{equation}\label{eq:figure-eight-polynomial}
    R_{a,b}(x)=x^{a+b}-x^a-x^b-3.
  \end{equation}

  \item If \(C\) is a theta graph with path lengths \(a,b,c\geq1\), then
  \begin{equation}\label{eq:theta-pressure}
    \frac{1}{\alpha^a+1}
    +\frac{1}{\alpha^b+1}
    +\frac{1}{\alpha^c+1}=1,
  \end{equation}
  and \(\alpha\) is the unique root \(x>1\) of
  \begin{equation}\label{eq:theta-polynomial}
    \Theta_{a,b,c}(x)
    =x^{a+b+c}-x^a-x^b-x^c-2.
  \end{equation}

  \item If \(C\) is a dumbbell with cycle lengths \(a,b\geq1\) and corridor
  length \(c\geq1\), then
  \begin{equation}\label{eq:dumbbell-pressure}
    \alpha^{2c}(\alpha^a-1)(\alpha^b-1)=4,
  \end{equation}
  and \(\alpha\) is the unique root \(x>1\) of
  \begin{align}
    D_{a,b,c}(x)
    &=
    x^{a+b+2c}-x^{a+2c}-x^{b+2c}+x^{2c}-4.
    \label{eq:dumbbell-polynomial}
  \end{align}
\end{enumerate}
Conversely, the unique root \(x>1\) in each of these three families is the
Hashimoto spectral radius of the corresponding rank-two core.
\end{theorem}

\begin{proof}
The kernel classification preceding the theorem proves that the three cases
are exhaustive and mutually exclusive at the level of kernel type.

\emph{Figure-eight.}  For a figure-eight, the weighted turn
matrix commutes with the involution that exchanges the two orientations of
each loop.  Perron uniqueness therefore makes the positive Perron vector
invariant under this involution, so it takes the same
value on the two orientations of each loop.  Applying
\cref{thm:kernel-pressure} and writing the two reversal-invariant coordinates
gives \eqref{eq:figure-eight-pressure} directly.
Clearing denominators yields both \eqref{eq:figure-eight-product} and
\eqref{eq:figure-eight-polynomial}.  The left side of
\eqref{eq:figure-eight-pressure} is strictly decreasing from \(2\) to \(0\)
on \(x\in(1,\infty)\), so the root is unique.

\emph{Theta graph.}  For a theta graph, after traversing one of the three
paths, a non-backtracking path may return along either of the other two paths.
Put \(z=x^{-1}\).  The involution that interchanges the two branch vertices
and reverses all three oriented kernel edges commutes with the weighted turn
operator.  After normalization, Perron uniqueness makes its positive Perron
vector invariant under this involution.  The two orientations of each kernel
edge therefore carry the same Perron coordinate.  On the three reversal
pairs, the weighted turn equation is represented by the matrix with
off-diagonal entry \(z^{\ell_j}\), where
\((\ell_1,\ell_2,\ell_3)=(a,b,c)\), in column \(j\) and zero
diagonal.  The matrix determinant lemma, applied to
\[
  I+\operatorname{diag}(z^a,z^b,z^c)
  -\mathbf1_3(z^a,z^b,z^c),
\]
where \(\mathbf1_3\) is the three-dimensional all-ones column vector, shows
that its Perron root is one precisely when
\eqref{eq:theta-pressure} holds.  Indeed, under this scalar condition
the vector with coordinates \((1+z^{\ell_j})^{-1}\) is a strictly positive
eigenvector of eigenvalue one; irreducibility and Perron--Frobenius theory
then show that one is the spectral radius.  Clearing denominators gives
\eqref{eq:theta-polynomial}.  The left side of
\eqref{eq:theta-pressure} decreases strictly from \(3/2\) to zero.

\emph{Dumbbell graph.}  For the dumbbell, write
\(p=x^{-a}\), \(s_0=x^{-b}\), and \(w=x^{-c}\).
The weighted turn operator is invariant under independently exchanging
the two orientations of either loop.  Let \(A\) and \(B\) be the
Perron-vector values on the two loop-orientation pairs, and let \(U,V\) be the
values on the two corridor orientations.  The eigenvalue-one equations for
the weighted turn matrix are
\[
  (1-p)A=wU,\qquad V=2pA,\qquad
  (1-s_0)B=wV,\qquad U=2s_0B.
\]
Eliminating \(A,B,U,V\) gives
\[
  (1-p)(1-s_0)=4w^2ps_0,
\]
which is \eqref{eq:dumbbell-pressure}.  Its left side in the product form
\[
  x^{2c}(x^a-1)(x^b-1)=4
\]
is strictly increasing from zero to infinity for \(x>1\).  Expanding gives
\eqref{eq:dumbbell-polynomial}.  The converse in every case follows by
constructing the indicated subdivision and applying
\cref{thm:kernel-pressure}.
\end{proof}

\begin{corollary}[Typed biregular rank-two realization]
\label{cor:rank-two-tree-actions}
Fix \(r,s\geq2\).  A rank-two core from
\cref{thm:rank-two-polynomials} is realizable inside a typed quotient of
\(\T_{r+1,s+1}\) precisely under the following conditions:
\[
\begin{array}{c|c|c}
\text{kernel}&\text{length condition}&\text{additional valency condition}\\
\hline
\text{figure-eight}&a\equiv b\equiv0\pmod2&\max\{r,s\}\geq3\\
\text{theta}&a\equiv b\equiv c\pmod2&\text{none}\\
\text{dumbbell}&a\equiv b\equiv0\pmod2&\text{none}.
\end{array}
\]
In each case the deck group is free of rank two, acts freely and
type-preservingly, and has critical exponent \(\log\alpha\), where
\(\alpha\) is the distinguished root in the corresponding polynomial
family.
\end{corollary}

\begin{proof}
The figure-eight has cycles of lengths \(a,b\), the theta graph has cycle
lengths \(a+b,a+c,b+c\), and the dumbbell has cycles of lengths \(a,b\).
Thus the displayed congruences are necessary and sufficient for
bipartiteness.  The unique figure-eight branching vertex has degree four,
so it can be assigned a type exactly when one ambient valency is at least
four.  Every theta or dumbbell branching vertex has degree three, allowed
for both types because \(r,s\geq2\).  In the theta case, even path lengths
give the same type at the two branch vertices and odd path lengths give
opposite types.  In the dumbbell case, the parity of \(c\) makes the same
choice.  Hence all admissible type maps exist, and
\cref{thm:finite-core-realization} completes the missing degrees without
changing the core or its Perron root.
\end{proof}

\begin{proof}[Proof of \cref{thm:intro-rank-two}]
Combine the exhaustive polynomial classification in
\cref{thm:rank-two-polynomials} with the necessary and sufficient typed
conditions in \cref{cor:rank-two-tree-actions}.
The final decidability assertion is supplied by
\cref{thm:effective-rank-two} below.
\end{proof}

\begin{corollary}[Regular rank-two specialization]
Setting \(r=s=q\), the theta and dumbbell families occur for every
\(q\geq2\), while the figure-eight family occurs for \(q\geq3\).  This is
the type-preserving rank-two classification on \(\T_{q+1}\).
\end{corollary}

\begin{remark}
At fixed rank, removing the type-preserving hypothesis is not a formal
corollary: the connected canonical bipartite double cover of a connected
non-bipartite rank-\(g\) core has rank \(2g-1\).  Thus
\(\cA_{q,q}=\cA_q\) holds for the unrestricted finite-rank union, but a
rank-two untyped classification also includes the non-bipartite length
patterns in \cref{thm:rank-two-polynomials}.
\end{remark}

\subsection{An effective rank-two membership test}

The preceding classification turns the inverse problem into a finite exact
calculation for a prescribed algebraic number.

\begin{theorem}[Effective typed rank-two test]\label{thm:effective-rank-two}
Fix \(r,s\geq2\).  Let \(\alpha>1\) be a real algebraic integer, specified
by its minimal polynomial and a rational isolating interval, meaning an
interval containing \(\alpha\) and no other real zero of that polynomial.
Whether \(\alpha\) is
the Hashimoto spectral radius of a rank-two typed core for
\(\T_{r+1,s+1}\) is decidable by a finite search.  In addition to the
length and valency conditions of \cref{cor:rank-two-tree-actions}, it is
enough to use the following bounds:
\begin{enumerate}[label=\textup{(\roman*)}]
  \item in the figure-eight case it is enough to test positive integers
  \(a,b\) satisfying
  \begin{equation}\label{eq:figure-eight-bound}
    a,b\leq
    \frac{\log\!\left(1+4/(\alpha-1)\right)}{\log\alpha};
  \end{equation}
  \item in the dumbbell case it is enough to impose the same bound on
  \(a,b\) and
  \begin{equation}\label{eq:dumbbell-c-bound}
    c\leq
    \frac{1}{2\log\alpha}
    \log\!\left(\frac{4}{(\alpha-1)^2}\right);
  \end{equation}
  \item in the theta case, after ordering \(a\leq b\leq c\), it is enough
  to test
  \begin{equation}\label{eq:theta-ab-bound}
    \alpha^a\leq2,
    \qquad
    \alpha^b\leq1+\frac{2}{\alpha^a}<3.
  \end{equation}
  For each remaining pair \((a,b)\), the value of \(\alpha^c\), if it
  exists, is forced by
  \begin{equation}\label{eq:theta-c-forced}
    \alpha^c
    =
    \left(
      1-\frac{1}{\alpha^a+1}
       -\frac{1}{\alpha^b+1}
    \right)^{-1}-1.
  \end{equation}
  The possible integer \(c\) is then located by finitely many exact order
  comparisons in \(\QQ(\alpha)\), as described in the proof.
\end{enumerate}
All polynomial identities, inequalities, and congruence conditions are
decidable exactly in this finite search.
\end{theorem}

\begin{proof}
In the figure-eight case,
\((\alpha^a-1)(\alpha^b-1)=4\), and each factor is at least
\(\alpha-1\).  Hence
\[
  \alpha^a-1,\ \alpha^b-1\leq\frac4{\alpha-1},
\]
which proves \eqref{eq:figure-eight-bound}.  In the dumbbell case,
\[
  \alpha^{2c}(\alpha^a-1)(\alpha^b-1)=4.
\]
The lower bounds \(\alpha^a-1,\alpha^b-1\geq\alpha-1\) give
\eqref{eq:dumbbell-c-bound}; discarding the factor \(\alpha^{2c}\geq1\)
gives the stated bounds for \(a,b\).

For the exact search, the logarithmic displays are used only as concise
numerical bounds; no logarithm of an algebraic number needs to be evaluated.
In the figure-eight and dumbbell cases, enumerate positive integers \(n\)
while
\begin{equation}\label{eq:effective-ab-algebraic-bound}
  \alpha^n-1\leq\frac{4}{\alpha-1}.
\end{equation}
Once this inequality fails, it fails for every larger \(n\).  In the
dumbbell case, enumerate corridor lengths while
\begin{equation}\label{eq:effective-c-algebraic-bound}
  \alpha^{2n}(\alpha-1)^2\leq4.
\end{equation}
These are exact order comparisons in \(\mathbb Q(\alpha)\) and produce
precisely the finite ranges described in
\eqref{eq:figure-eight-bound}--\eqref{eq:dumbbell-c-bound}.

For a theta graph, set \(t_n=(\alpha^n+1)^{-1}\).  If
\(a\leq b\leq c\), then \(t_a\geq t_b\geq t_c\) and
\(t_a+t_b+t_c=1\).  Thus \(t_a\geq1/3\), which gives
\(\alpha^a\leq2\).  Moreover,
\[
  1-t_a=t_b+t_c\leq2t_b,
\]
so
\[
  \alpha^b+1\leq\frac2{1-t_a}
  =2+\frac2{\alpha^a}.
\]
This is \eqref{eq:theta-ab-bound}, and the remaining equation uniquely
forces \eqref{eq:theta-c-forced}.  To turn this uniqueness into a finite
algorithm, put
\[
  q_{a,b}
  =1-\frac1{\alpha^a+1}-\frac1{\alpha^b+1},
  \qquad
  X_{a,b}=q_{a,b}^{-1}-1.
\]
If \(q_{a,b}\leq0\) or \(X_{a,b}\leq1\), there is no positive integer
\(c\).  Otherwise, compare the algebraic numbers
\(\alpha,\alpha^2,\ldots\) successively with \(X_{a,b}\), and stop at the
first \(N\) for which \(\alpha^N>X_{a,b}\).  This process terminates because
\(\alpha>1\).  Any solution of \(\alpha^c=X_{a,b}\) must then have
\(1\leq c<N\), so the finitely many remaining powers can be tested exactly;
one also retains the ordering \(b\leq c\) and the required parity condition.

There are therefore finitely many integer tuples to inspect.  Restricting
them by the parity and figure-eight valency conditions in
\cref{cor:rank-two-tree-actions} gives exactly the typed biregular
candidates.  Arithmetic and order in the real algebraic number field
\(\QQ(\alpha)\) are effective: reduction modulo the minimal polynomial gives
exact field arithmetic, and the isolating interval permits exact sign and
order determination.  Thus every polynomial identity, power identity, and
inequality in the finite search can be checked exactly; see
\cite{BasuPollackRoy2006} for standard algorithms for real algebraic numbers
and real-root isolation.
\end{proof}

\begin{corollary}[Polynomial partial converse]\label{cor:rank-two-converse}
Fix \(r,s\geq2\), and let \(m_\alpha(x)\in\ZZ[x]\) be the minimal
polynomial of a real algebraic integer \(\alpha>1\).  The number \(\alpha\)
has a typed rank-two Hashimoto realization for \(\T_{r+1,s+1}\) if and only
if the finite search in
\cref{thm:effective-rank-two} finds parameters for which
\[
  m_\alpha(x)\mid R_{a,b}(x),\qquad
  m_\alpha(x)\mid\Theta_{a,b,c}(x),\qquad\text{or}\qquad
  m_\alpha(x)\mid D_{a,b,c}(x),
\]
and \(\alpha\) is the distinguished positive root, with the corresponding
conditions of \cref{cor:rank-two-tree-actions}.  Every number obtained in
this way is a weak Perron algebraic integer and is realized as
\(\exp(\delta_\Gamma)\) for a rank-two free type-preserving action on
\(\T_{r+1,s+1}\).
\end{corollary}

\section{Entropy strata beyond finite rank}
\label{sec:structural-problems}

\subsection{The unrestricted Hashimoto inverse problem}

The rank-two theorem is an exact partial converse, but it does not make the
weak-Perron condition sufficient at unbounded rank.  To isolate the extra
graph structure, let \(O_C\) be the origin-incidence matrix whose rows are
indexed by \(V(C)\), whose columns are indexed by \(\vec E(C)\), and whose
\((v,e)\)-entry is one exactly when \(o(e)=v\).  Let \(J_C\) be the
permutation matrix of the fixed-point-free involution \(e\mapsto\bar e\), so
\((J_C)_{e,f}=1\) exactly when \(f=\bar e\).

\begin{proposition}[Hashimoto incidence factorization]
\label{prop:hashimoto-factorization}
For every finite multigraph \(C\),
\begin{equation}\label{eq:hashimoto-factorization}
  B_C=J_C(O_C^{\mathsf T}O_C-I_{2|E(C)|}).
\end{equation}
After the oriented edges are grouped by their origins,
\(O_C^{\mathsf T}O_C-I_{2|E(C)|}\) is the direct sum, over
\(v\in V(C)\), of the
matrices
\[
  \mathbf1_{\deg(v)}\mathbf1_{\deg(v)}^{\mathsf T}-I_{\deg(v)}.
\]
Here \(\mathbf1_j\) denotes the all-ones column vector of length \(j\).
Conversely, a set partition into such origin blocks together with a
fixed-point-free reversal involution reconstructs a multigraph and its
Hashimoto matrix by \eqref{eq:hashimoto-factorization}.
\end{proposition}

\begin{proof}
For oriented edges \(e,f\),
\[
  \bigl(J_CO_C^{\mathsf T}O_C\bigr)_{e,f}=1
  \quad\Longleftrightarrow\quad
  o(\bar e)=o(f)
  \quad\Longleftrightarrow\quad
  t(e)=o(f).
\]
Subtracting \(J_C\) deletes precisely the transition \(e\to\bar e\).
The block description and the converse are immediate from the definition of
the origin-incidence matrix.
\end{proof}

Let
\[
  \mathcal W_{r,s}
  =\{1\}\cup\{\lambda\in(1,\sqrt{rs}]:
      \lambda\text{ is a weak Perron algebraic integer}\}.
\]
Then \(\cA_{r,s}\subseteq\mathcal W_{r,s}\).  Lind characterized Perron
numbers as spectral radii of primitive nonnegative integral matrices
\cite{Lind1984}; Boyle--Handelman developed a
far more general inverse spectral theory for nonnegative matrices
\cite{BoyleHandelman1991}.  Neither result supplies the reversal involution
and complete local turn blocks in
\eqref{eq:hashimoto-factorization}.

\begin{problem}[Inverse Ihara--Bass realization]\label{prob:inverse-ihara}
Fix \(r,s\geq2\).  Characterize the monic irreducible polynomials
\(p(x)\in\ZZ[x]\) for which there is a finite connected bipartite core
\(C=C_0\sqcup C_1\), with
\[
  2\leq\deg_C(v)\leq r+1\ (v\in C_0),
  \qquad
  2\leq\deg_C(v)\leq s+1\ (v\in C_1),
\]
such that
\begin{enumerate}[label=(\roman*)]
  \item \(p(x)\) divides
  \(\det(x^2I_{|V(C)|}-xA_C+(D_C-I_{|V(C)|}))\);
  \item \(\rho(B_C)>1\) is a zero of \(p(x)\).
\end{enumerate}
In an equivalent matrix formulation, one must characterize those
nonnegative integral matrices whose spectral data contain \(p\) and which,
after a simultaneous permutation of rows and columns, admit a set partition
into origin blocks and a fixed-point-free reversal involution satisfying
\eqref{eq:hashimoto-factorization}, together with the bipartite and two-sided
degree bounds above.  The factorization is automatic once a graph \(C\) has
been given; its force is as an additional structural constraint in the
matrix-only inverse problem.
Determine, in particular, which elements of \(\mathcal W_{r,s}\) belong to
\(\cA_{r,s}\).
\end{problem}

For fixed \(g\), \cref{thm:fixed-rank-kernels} replaces enumeration by
vertex number with finitely many typed kernels and admissible length vectors.
At rank two, \cref{thm:effective-rank-two} makes the test finite for each
prescribed algebraic number.  The unrestricted problem asks whether
comparable bounds or structural criteria exist uniformly as
\(g\to\infty\).  Setting \(r=s=q\) and forgetting rank recovers the regular
inverse problem through \cref{cor:regular-finite-core-spectrum}.

\subsection{Rank strata and degeneration boundaries}

\Cref{thm:rank-drop} shows that a nontrivial accumulation of fixed-rank
finite-state growth rates is forced onto a lower-rank stratum.  A natural next
step is to determine the converse under a fixed degree bound: which elements
of \(\mathcal R_{g_0;r,s}\) can be approached by elements of
\(\mathcal R_{g;r,s}\) when \(g_0<g\)?  The answer depends on whether the
realizing lower-rank core has spare degree at suitable vertices of the
required type.  A complete statement should therefore refine rank by the
type and number of available ports.  Such a refinement may lead to a
Cantor--Bendixson description of the closure and derived-set filtration of
the union of the fixed-rank spectra.

There is also a multi-scale moduli question behind
\cref{thm:rank-drop}.  Fix a kernel \(K\), let \(\ell_n\) be a sequence of
positive integral length vectors, and normalize by
\(\|\ell_n\|_1=\sum_{e\in E(K)}\ell_n(e)\).  A positive coordinate in a
projective limit records an edge whose length grows on the total-length scale;
for each fixed \(x>1\), the corresponding columns of the weighted turn matrix
then vanish.  A zero coordinate, however, can represent either a bounded
length or an unbounded length growing on a smaller scale, and therefore cannot
be identified directly with an edge contraction.  A useful compactification
should consequently combine the projective face with a flag recording
successive growth scales.  At each scale one would delete the vanishing
columns, pass to recurrent components, and suppress forced degree-two chains,
iterating the rank-drop procedure.  Different flags may still produce the
same lower-rank core.  Recording the limiting Perron root together with the
corresponding normalized Perron eigenvectors, or equivalently the resulting
non-backtracking boundary measures, would refine the scalar rank-drop theorem
to a compactification of the underlying dynamics.

\subsection*{Declaration of generative AI use}
During the preparation of this manuscript, the author used ChatGPT to assist with language editing, readability, organization, and preliminary consistency checks. The author independently reviewed and verified the resulting text and mathematical content, made all necessary revisions, and takes full responsibility for the content of the manuscript.
% =============================================================================

\end{document}